\documentclass[reqno]{amsart}
\usepackage{amsfonts,amssymb,verbatim,amsmath,amsthm,latexsym,textcomp,amscd}
\usepackage{latexsym,amsfonts,amssymb,epsfig,verbatim}
\usepackage{amsmath,amsthm,amssymb,latexsym,graphics,textcomp}
\usepackage{graphicx}
\usepackage{color}
\usepackage{url}
\usepackage{psfrag}

\begin{document}

\newtheorem{theorem}{Theorem}[section]
\newtheorem{prop}[theorem]{Proposition}
\newtheorem{lemma}[theorem]{Lemma}
\newtheorem{cor}[theorem]{Corollary}
\newtheorem{definition}[theorem]{Definition}
\newtheorem{conj}[theorem]{Conjecture}
\newtheorem{claim}[theorem]{Claim}
\newtheorem{qn}[theorem]{Question}
\newtheorem{defn}[theorem]{Definition}
\newtheorem{defth}[theorem]{Definition-Theorem}
\newtheorem{obs}[theorem]{Observation}
\newtheorem{rmk}[theorem]{Remark}
\newtheorem{ans}[theorem]{Answers}
\newtheorem{slogan}[theorem]{Slogan}

\newcommand{\boundary}{\partial}
\newcommand{\hhat}{\widehat}
\newcommand{\C}{{\mathbb C}}
\newcommand{\Ga}{{\Gamma}}
\newcommand{\G}{{\Gamma}}
\newcommand{\s}{{\Sigma}}
\newcommand{\PSL}{{PSL_2 (\mathbb{C})}}
\newcommand{\pslc}{{PSL_2 (\mathbb{C})}}
\newcommand{\pslr}{{PSL_2 (\mathbb{R})}}
\newcommand{\Gr}{{\mathcal G}}
\newcommand{\integers}{{\mathbb Z}}
\newcommand{\natls}{{\mathbb N}}
\newcommand{\ratls}{{\mathbb Q}}
\newcommand{\reals}{{\mathbb R}}
\newcommand{\proj}{{\mathbb P}}
\newcommand{\lhp}{{\mathbb L}}
\newcommand{\tube}{{\mathbb T}}
\newcommand{\cusp}{{\mathbb P}}
\newcommand\AAA{{\mathcal A}}
\newcommand\HHH{{\mathbb H}}
\newcommand\BB{{\mathcal B}}
\newcommand\CC{{\mathcal C}}
\newcommand\DD{{\mathcal D}}
\newcommand\EE{{\mathcal E}}
\newcommand\FF{{\mathcal F}}
\newcommand\GG{{\mathcal G}}
\newcommand\HH{{\mathcal H}}
\newcommand\II{{\mathcal I}}
\newcommand\JJ{{\mathcal J}}
\newcommand\KK{{\mathcal K}}
\newcommand\LL{{\mathcal L}}
\newcommand\MM{{\mathcal M}}
\newcommand\NN{{\mathcal N}}
\newcommand\OO{{\mathcal O}}
\newcommand\PP{{\mathcal P}}
\newcommand\QQ{{\mathcal Q}}
\newcommand\RR{{\mathcal R}}
\newcommand\SSS{{\mathcal S}}
\newcommand\TT{{\mathcal T}}
\newcommand\UU{{\mathcal U}}
\newcommand\VV{{\mathcal V}}
\newcommand\WW{{\mathcal W}}
\newcommand\XX{{\mathcal X}}
\newcommand\YY{{\mathcal Y}}
\newcommand\ZZ{{\mathcal Z}}
\newcommand\CH{{\CC\Hyp}}
\newcommand{\Chat}{{\hat {\mathbb C}}}
\newcommand\MF{{\MM\FF}}
\newcommand\PMF{{\PP\kern-2pt\MM\FF}}
\newcommand\ML{{\MM\LL}}
\newcommand\PML{{\PP\kern-2pt\MM\LL}}
\newcommand\GL{{\GG\LL}}
\newcommand\Pol{{\mathcal P}}
\newcommand\half{{\textstyle{\frac12}}}
\newcommand\Half{{\frac12}}
\newcommand\Mod{\operatorname{Mod}}
\newcommand\Area{\operatorname{Area}}
\newcommand\ep{\epsilon}
\newcommand\Hypat{\widehat}
\newcommand\Proj{{\mathbf P}}
\newcommand\U{{\mathbf U}}
 \newcommand\Hyp{{\mathbf H}}
\newcommand\D{{\mathbf D}}
\newcommand\Z{{\mathbb Z}}
\newcommand\R{{\mathbb R}}
\newcommand\Q{{\mathbb Q}}
\newcommand\E{{\mathbb E}}
\newcommand\EXH{{ \EE (X, \HH_X )}}
\newcommand\EYH{{ \EE (Y, \HH_Y )}}
\newcommand\GXH{{ \GG (X, \HH_X )}}
\newcommand\GYH{{ \GG (Y, \HH_Y )}}
\newcommand\ATF{{ \AAA \TT \FF }}
\newcommand\PEX{{\PP\EE  (X, \HH , \GG , \LL )}}
\newcommand{\lct}{\Lambda_{CT}}
\newcommand{\lel}{\Lambda_{EL}}
\newcommand{\lgel}{\Lambda_{GEL}}
\newcommand{\lre}{\Lambda_{\mathbb{R}}}

\newcommand\til{\widetilde}
\newcommand\length{\operatorname{length}}
\newcommand\tr{\operatorname{tr}}
\newcommand\gesim{\succ}
\newcommand\lesim{\prec}
\newcommand\simle{\lesim}
\newcommand\simge{\gesim}
\newcommand{\simmult}{\asymp}
\newcommand{\simadd}{\mathrel{\overset{\text{\tiny $+$}}{\sim}}}
\newcommand{\ssm}{\setminus}
\newcommand{\diam}{\operatorname{diam}}
\newcommand{\pair}[1]{\langle #1\rangle}
\newcommand{\T}{{\mathbf T}}
\newcommand{\I}{{\mathbf I}}

\newcommand{\tw}{\operatorname{tw}}
\newcommand{\base}{\operatorname{base}}
\newcommand{\trans}{\operatorname{trans}}
\newcommand{\rest}{|_}
\newcommand{\bbar}{\overline}
\newcommand{\UML}{\operatorname{\UU\MM\LL}}
\newcommand{\EL}{\mathcal{EL}}
\newcommand{\qle}{\lesssim}

\newcommand\Gomega{\Omega_\Gamma}
\newcommand\nomega{\omega_\nu}

\title{Cannon-Thurston Maps}

\author{Mahan Mj}
\address{School
of Mathematics, Tata Institute of Fundamental Research. 1, Homi Bhabha Road, Mumbai-400005, India}

\email{mahan@math.tifr.res.in}

\thanks{Research partly supported by a DST JC Bose Fellowship.}

\date{\today}

 \begin{abstract}
	We give an overview of the theory of Cannon-Thurston maps which  forms one of the  links between the complex analytic and  hyperbolic geometric study of Kleinian groups. We also briefly sketch connections to hyperbolic subgroups of hyperbolic groups and end with some open questions.
\end{abstract}

\maketitle

%\tableofcontents
\section{Kleinian groups and limit sets}\label{kg} The Lie group $\PSL$ can be viewed from three different points of view:
\begin{enumerate}
	\item As a Lie group or a matrix group ({\bf group-theoretic}).
	\item As the isometry group of hyperbolic 3-space $\Hyp^3$--the upper half space $\{(x,y,z): z>0\}$ equipped with the metric   $ds^2 = \frac{dx^2 + dy^2 + dz^2}{z^2}$, or equivalently the open ball
	$\{(x,y,z): (x^2 + y^2+z^2) = r^2 < 1\}$ equipped with the metric   $ds^2 = \frac{4(dx^2 + dy^2 + dz^2)}{(1-r^2)^2}$ ({\bf geometric}).
	\item As the group of M\"obius transformations   of the Riemann sphere $\hat{\mathbb{C}}$ ({\bf complex dynamic/analytic}).
\end{enumerate}

A finitely generated discrete subgroup $\Ga \subset \PSL$ is called a {\bf Kleinian group}. Depending on how we decide to look at $\PSL$, the group $\Gamma$ can  accordingly be thought of as a discrete subgroup of a Lie group, as the fundamental  group of the complete hyperbolic 3-manifold $\Hyp^3/\Gamma$, or in terms of its action by  holomorphic automorphisms  of  $\hat{\mathbb{C}}$. If $\Ga$ is not virtually abelian, it is called {\bf non-elementary}. Henceforth, unless explicitly stated otherwise, we shall assume that {\it all Kleinian groups in this article are non-elementary}. If $\Gamma$ can be conjugated by an element of $\pslc$ to be contained in $\pslr$ it is referred to as a {\bf Fuchsian group}\footnote{Both Fuchsian and Kleinian groups were discovered by Poincar\'e.}. If $\Gamma$ is abstractly isomorphic to $\pi_1(S)$, the fundamental group of a closed surface $S$, we shall refer to it as a {\bf surface Kleinian group}.

In the 1960's Ahlfors and Bers studied deformations of Fuchsian surface groups in  $\PSL$, giving rise to the theory of  quasi-Fuchsian groups. Their techniques were largely complex analytic in nature. 
In the 1970's and 80's, the field was revolutionized by Thurston, who brought in a rich and varied set of techniques from three-dimensional hyperbolic geometry.   A conjectural picture of the deep relationships  between the analytic and geometric points of view was outlined by Thurston in his visionary paper \cite{thurston-bams}.
Perhaps the most well-known problem (predating Thurston) in this line of study was the Ahlfors' measure zero conjecture, resolved in the last decade by Brock-Canary-Minsky \cite{minsky-elc2}. Another (more topological) well-known problem that also predates Thurston asks if limit sets of Kleinian groups are locally connected. This is the specific problem that will concern us here. We need to fix some  terminology and notation first. Identify
the Riemann sphere $\hat{\mathbb{C}}$  with the sphere at infinity $S^2$ of $\Hyp^3$.  Thus, $S^2$ encodes the `ideal' boundary of ${\mathbb{H}}^3$, consisting of asymptote classes of geodesics. 
By adjoining $S^2$ to $\Hyp^3$, we obtain the closed 3-ball $\D^3$. The topology on $S^2$ is the usual one induced by the round
metric given by the   angle subtended at the origin $0 \in \D^3$.  The
geodesics turn out to be semicircles meeting the boundary $S^2$ at right angles.

\begin{definition}
The {\bf limit set} $\Lambda_\Ga$ of the Kleinian group $\Gamma$ is the collection of accumulation points of a $\Gamma$-orbit $\Gamma\cdot z$ for some (any) $z \in \hat{\mathbb{C}}$.
\end{definition}
 The limit set  $\Lambda_\Gamma$ is independent of $z$ and may be regarded as the locus of chaotic dynamics of $\Ga$ on $\hat{\mathbb{C}}$. For non-elementary $\Ga$ and any $z \in \Lambda_\Gamma$, $\Gamma \cdot z$ is dense  in $\Lambda_\Gamma$.  Hence $\Lambda_\Gamma$ is the smallest closed non-empty $\Ga-$invariant subset of $ \hat{\mathbb{C}}$.
 If we take $z \in \Hyp^3$ instead, then the collection of accumulation points of any $\Gamma$-orbit $\Gamma\cdot z \subset \D^3$ is also $\Lambda_\Ga$.

 \begin{definition}
 The complement of the limit set $\hat{\mathbb{C}} \setminus \Lambda_\Ga$ is called the domain of discontinuity $\Omega_\Ga$ of $\Ga$. 
\end{definition}
The Kleinian group $\Ga$ acts freely and properly discontinuously on $\Omega_\Ga$ and  $\Omega_\Ga/\Ga$ is a (possibly disconnected) Riemann surface. 

\subsection{Fuchsian and Quasi-Fuchsian Groups} We first give an explicit example of a Fuchsian group.

\noindent {\bf An example of a Fuchsian group:} Consider the standard identification space description of the genus two orientable surface $\Sigma_2$ as
 an octagon with  edge labels $a_1, b_1, a_1^{-1}, b_1^{-1}, a_2, b_2,
a_2^{-1}, b_2^{-1}$. A hyperbolic metric 
 on $\Sigma_2$ is one where   each point has a small neighborhood isometric to a small disk in 
$\Hyp^2$. By the Poincar\'e  polygon theorem, it suffices to find  a regular hyperbolic octagon (with all sides equal and all
angles equal) with each interior angle equal to $\frac{2\pi}{8} $. Now, the infinitesimal regular
octagon at the tangent space to the origin has interior angles equal to $\frac{3\pi}{4} $. Also the ideal 
regular
octagon in ${\mathbb{H}}^2$ has all interior angles zero. See figure below.

\begin{center}
	
	\includegraphics[height=5cm]{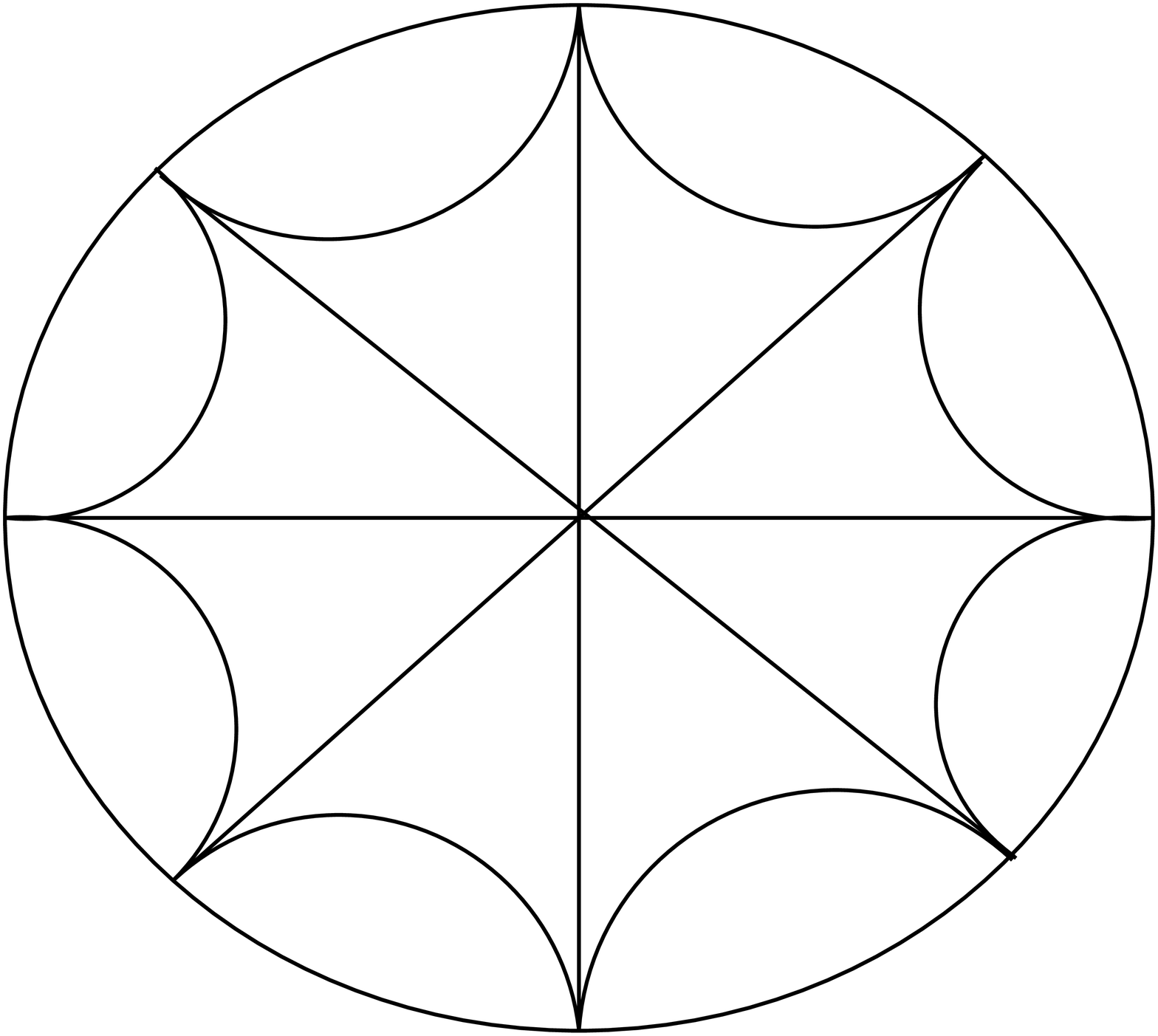}

\end{center}

By the Intermediate Value Theorem, there exists an intermediate regular octagon
with all interior angles  equal to $\frac{\pi}{4} $. The group $G$ that results from side-pairing transformations
corresponds to a Fuchsian group, or equivalently, a discrete faithful representation $\rho$ of $\pi_1(\Sigma_2)$ into $\pslr$. Equivalently we may 
regard  $\rho$ as a representation of $\pi_1(\Sigma_2)$ into $\pslc$ with image that can be conjugated to lie in $\pslr$. Alternately, $\rho(\pi_1(\Sigma_2))$ preserves a totally geodesic plane in $\Hyp^3$. The limit set of $G=\rho(\pi_1(\Sigma_2))$ is then a round circle.\\

\noindent {\bf Quasi-Fuchsian groups:} If we require the limit set to be only topologically a circle, i.e.\  a Jordan curve, then we obtain a more general class of Kleinian groups:

\begin{defn}
Let $\rho: \pi_1(S) \to \pslc$ be a discrete faithful representation such that the limit set of $G= \rho(\pi_1(S))$ is a Jordan curve in $S^2$. Then $G$ is said to be {\bf quasi-Fuchsian}. The collection of conjugacy classes of quasi-Fuchsian with the complex analytic structure inherited from $\pslc$ is denoted as $QF(S)$. 
\end{defn}

The domain of discontinuity $\Omega$ of a  quasi-Fuchsian $G$ consists of two open invariant topological disks $\Omega_1, \Omega_2$ in $\hat{\mathbb C}$. Hence the quotient $\Omega/G$ is the disjoint union $\Omega_1/G \sqcup \Omega_2/G$ and we have a map $\tau: QF(S) \to Teich(S) \times Teich(S)$, where $Teich(S)$ denotes the Teichm\"uller space of $S$--the space of marked hyperbolic (or complex) structures on $S$.
The { \bf Bers simultaneous Uniformization Theorem} asserts:

\begin{theorem}
$\tau: QF(S) \to Teich(S) \times Teich(S)$ is bi-holomorphic.
\end{theorem}
Thus, given any two conformal structures $T_1, T_2$ on a surface, there is a unique discrete quasi-Fuchsian  $G$ (up to conjugacy) whose limit set $\Lambda_G$ is  topologically a circle, and the quotient of whose domain
of discontinuity is $T_1 \sqcup T_2$. See figure below \cite{kabaya}, where the inside and the outside of the Jordan curve correspond to $\Omega_1, \Omega_2$.

\begin{center}
	
	\includegraphics[height=5cm]{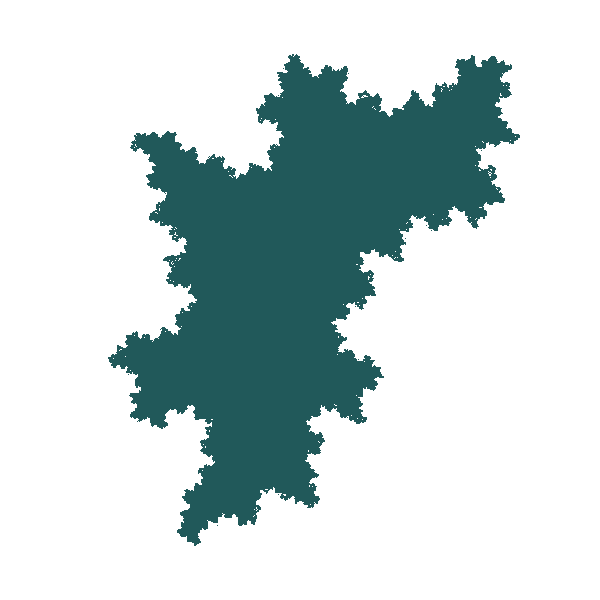}

\end{center}

\subsection{Degenerate groups and the Ending Lamination Theorem} \label{elt}
Quasi-Fuchsian groups were studied by Ahlfors and Bers analytically as deformations of Fuchsian groups. Thurston  \cite{thurstonnotes} introduced  a new set of geometric techniques in the study.

\begin{defn}
The {\bf convex hull} $CH_G$ of  $\Lambda_G$ is the smallest non-empty closed convex subset of ${\mathbb{H}}^3$ invariant under $G$. 

Let $M = \Hyp^3/G$.
The quotient of $CH_G$ by $G$ is
called the {\it convex core $CC(M)$} of $M$.
\end{defn}

The  convex hull $CH_G$ can be constructed by joining all distinct pairs of points on $\Lambda_G$ by bi-infinite geodesics, iterating this
construction and finally taking the closure. It can also be described as the closure of the union of all ideal tetrahedra, whose vertices lie in $\Lambda_G$. The convex core $CC(M)$ is homeomorphic to $S \times [0,1]$.

The distance between the boundary components $S \times \{ 0 \}$
and $S \times \{ 1 \}$ in the convex core $CC(M)$, measured with respect to the hyperbolic metric, is a crude geometric measure of the complexity of the quasi Fuchsian group $G$. We shall call it the {\bf thickness} $t_G$ of $CC(M)$, or of the quasi Fuchsian group $G$.  (We note here parenthetically that the notions of convex hull $CH_G$ and convex core $CC(M)$ go through for any Kleinian group $G$ and the associated complete hyperbolic manifold $\Hyp^3/G$.) For quasi-Fuchsian groups, we ask:

\begin{qn}\label{thickness}
	What happens as thickness tends to infinity?
\end{qn}

To address this question more precisely we need to introduce a topology on the space of representations.
\begin{definition}\label{def-altop}
A sequence of representations $\rho_n : \pi_1(S) \to \pslc$ is said to converge {\bf algebraically} to $\rho_\infty : \pi_1(S) \to \pslc$ if for all $g \in \pi_1(S)$, $\rho_n (g) \to \rho_\infty(g)$ in $\pslc$. The collection of conjugacy classes of discrete faithful representations of  $\pi_1(S)$ into $\pslc$ equipped with the algebraic topology is denoted as $AH(S)$.
\end{definition}

It is not even clear a priori that, as $t_G$ tends to infinity, limits exist in $AH(S)$. However, Thurston's 
{\bf Double Limit Theorem} \cite{thurston-hypstr2, kapovich-book, otal-book} guarantees that if we have a sequence $G_n$ with thickness $t_{G_n}$ tending to infinity, subsequential limits (in the space of discrete faithful representations with a suitable topology, see Definition \ref{topofconv})  do in fact exist.\\

\noindent {\bf Geodesic Laminations:}

\begin{defn} A geodesic lamination on a hyperbolic surface is a foliation of a closed subset with geodesics, i.e.\ it is a closed set given as a  disjoint union of geodesics (closed or bi-infinite). \end{defn} A geodesic lamination on a surface may further be equipped with a transverse (non-negative) measure to obtain a {\bf measured lamination}. The space of measured (geodesic) laminations on $S$ is then a positive cone in a vector space and is denoted as $\ML (S)$. It can be projectivized to obtain the space of projectivized measured  laminations $\PML(S)$. It was shown by Thurston \cite{FLP} that $\PML(S)$ is homeomorphic to a sphere and can be adjoined to $Teich(S)$ compactifying it to a closed ball.

\begin{defn} \cite[Definition 8.8.1]{thurstonnotes} A {\bf pleated surface} in a hyperbolic three-manifold $N$ is
	a complete hyperbolic surface $S$ of finite area, together with an isometric map $f :
	S \to N$ such that every $x \in S$ is in the interior of some geodesic segment (in $S$) which
	is mapped by $f$ to a geodesic segment (in $N$). Also, $f$ must take every cusp (corresponding to a maximal parabolic subgroup) of $S$ to a
	cusp of $N$
	
	The {\bf pleating locus} of the pleated surface $f: S \to M$  is the set $\gamma \subset S$ consisting of those points in the pleated surface which are in the
	interior of unique geodesic segments mapped to geodesic segments.
\end{defn}

\begin{prop} \cite[Proposition 8.8.2]{thurstonnotes} The pleating locus $\gamma$
	is a geodesic lamination on $S$. The map $f$ is totally geodesic in
	the complement of 
	$\gamma$. \end{prop}

Thurston further shows \cite[Ch. 8]{thurstonnotes} that the  boundary components  of the convex core $CC(M)$  are pleated surfaces.

As thickness (Question \ref{thickness}) tends to infinity, the convex core may converge (up to extracting  subsequences) to one of two kinds of convex manifolds $CC(M_\infty)$ (see schematic diagram below):

\begin{enumerate}
	\item $CC(M_\infty)$ is homeomorphic to $S \times [0, \infty)$. Here $S \times \{ 0\}$ is the single boundary component of the convex core  $CC(M_\infty)$.
	Such a manifold is called  {\bf simply degenerate} and the corresponding Kleinian group a  simply degenerate surface Kleinian group. In this case the limit set $\Lambda_G$ is a dendrite (topologically a tree) and the domain of discontinuity $\Gomega$ is homeomorphic to an open disk with $\Gomega/G$ a Riemann surface homeomorphic to $S \times \{ 0\}$. The end of $M_\infty$ facing $\Gomega/G$ is called a {\bf geometrically finite end}.
	\item $CC(M_\infty)$ is homeomorphic to $S \times (-\infty, \infty)$. Such a manifold is called  {\bf doubly degenerate} and the corresponding Kleinian group a  doubly degenerate surface Kleinian group. In this case the limit set $\Lambda_G$ is all of $S^2$ and the domain of discontinuity $\Gomega$ is empty.\\
\end{enumerate}

\begin{center}
	
	\includegraphics[height=6cm]{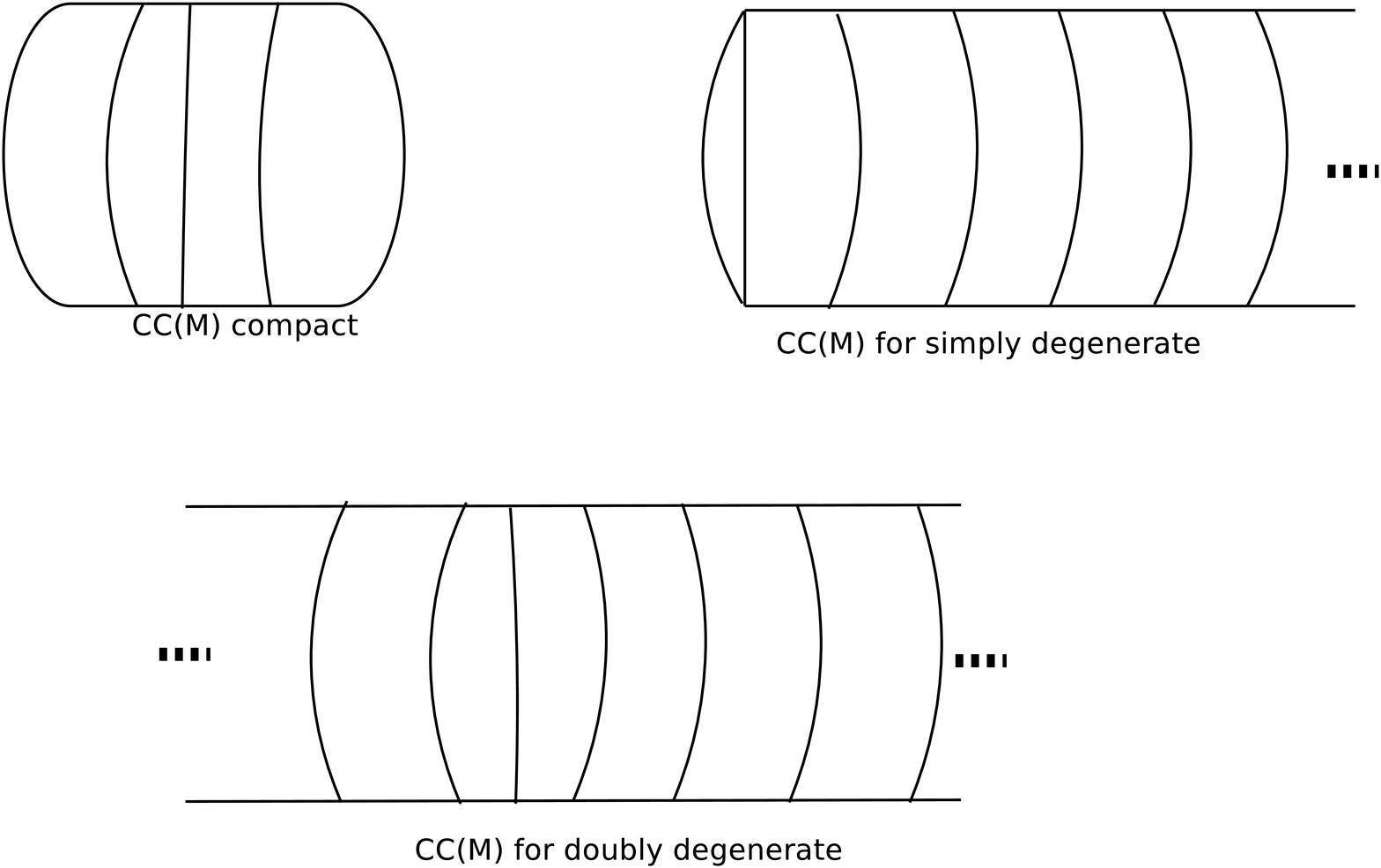}

\end{center}

The ends (one for simply degenerate and two for doubly degenerate) of $CC(M_\infty)$ are called the {\bf degenerate end(s)} of $M$. Thurston \cite{thurstonnotes} and Bonahon \cite{bonahon-bouts} show that any such degenerate end $E$ is {\bf geometrically tame}, i.e.\ there is a sequence of simple closed curves $\{ \sigma_n \}$ on $S$ such that their geodesic realizations in $CC(M_\infty)$ exit $E$ as $n$ tends to $\infty$.
Further  \cite[Ch. 9]{thurstonnotes}, the limit of any such exiting  sequence (in $\PML(S)$; 
the topology on $\PML(S)$ is quite close to the Hausdorff topology on $S$) is a unique lamination $\lambda$ called the {\bf ending lamination} of $E$.
Thus a doubly degenerate manifold has two ending laminations, one for each degenerate end, while a simply degenerate manifold has a geometrically finite end corresponding to a Riemann surface  $\Gomega/G (\in Teich(S))$ at infinity and a degenerate end corresponding to an ending lamination. 

These two pieces of information--Riemann surfaces at infinity and ending laminations--give the {\bf end-invariants} of $M$.
The ending lamination for a  geometrically infinite end does not depend on the reference hyperbolic structure on $S$. It may be regarded as a purely topological piece of data associated to such an end. 

{\it  We may thus think of the ending lamination as the analog, in the case of geometrically infinite ends, of the conformal structure at infinity for a  geometrically infinite end.}

The  Ending Lamination Theorem of Brock-Canary-Minsky then takes the place of the Bers' simultaneous uniformization theorem and asserts:
\begin{theorem} \cite{minsky-elc1, minsky-elc2} \label{elthm} Let 
$M$ be a simply or doubly degenerate manifold. Then $M$ is determined up to isometry by its end-invariants.
\end{theorem}
Thus, 
the Ending Lamination Theorem justifies the following and  may be considered an analog of Mostow Rigidity
for infinite covolume Kleinian groups.
\begin{slogan}
Topology implies Geometry.
\end{slogan}

In order to complete the picture of the Ahlfors-Bers theory to study degenerate surface Kleinian groups $G$ in terms of the dynamics of $G$ on the Riemann sphere $\hat{\mathbb{C}}$, the following issue remains to be addressed:

\begin{qn}\label{ctmq}
Can the data of the ending lamination(s) be extracted from the dynamics of $G$ on  $\hat{\mathbb{C}}$?
\end{qn}

In more informal terms, 
\begin{qn}
	Is the geometric object "at infinity" of the quotient manifold $\Hyp^3/\Gamma$ (i.e.\ the ending lamination) determined 	by the dynamics of $\G$ at infinity (i.e.\, the action of $\G$ on $S^2$)? 
\end{qn}

We will make these questions precise below.
The attempt to make Question \ref{ctmq} precise brings us to the following.

\section{Cannon-Thurston maps}\label{ctint}

\subsection{The main theorem for closed surface Kleinian groups}
 In \cite[Problem 14]{thurston-bams}, Thurston raised the following  question, which is at the heart of the work we discuss here:

\begin{qn} \label{thurston-qn}        Suppose $\G$ has the property that $(\Hyp^3 \cup \Gomega)/{\Gamma}$ is compact. Then is it
	true that the limit set of any other Kleinian group $\Gamma^{\prime}$ isomorphic to $\Gamma$ is the
	continuous image of the limit set of $\Gamma$, by a continuous map taking the
	fixed points of an(y) element $\gamma$ to the fixed points of the corresponding element $\gamma^\prime$? \end{qn}

A special case of Question \ref{thurston-qn} was answered affirmatively in seminal work of Cannon and Thurston \cite{CT, CTpub}:

\begin{theorem}\cite{CTpub}
Let $M$ be a closed hyperbolic 3-manifold fibering over the circle with 
fiber $\Sigma$. Let $\til \Sigma$ and $\til M$ denote the universal
covers of $F$ and $M$ respectively. After identifying  $\til \Sigma$ (resp. $\til M$) with $\Hyp^2$ (resp. $\Hyp^3$), we obtain the compactification  ${{\mathbb{D}}^2}=\Hyp^2\cup{{S}}^1$ (resp.
${{\mathbb{D}}^3}=\Hyp^3\cup{{S}}^2$) by attaching the circle $S^1$ (resp. the sphere $S^2$) at infinity.
Let $i: \Sigma \to M$ denote the  inclusion map of the fiber and $\til{i}: \til{\Sigma} \to \til{M}$ the lift to the universal cover. Then $\til i$
extends to a continuous map  $\hat{i}: \mathbb{D}^2 \to \mathbb{D}^3$.
\label{ctthm}
\end{theorem}

An amazing implication of Theorem \ref{ctthm} is that $\til \Sigma$ is an embedded disk in (the ball model) of $\Hyp^3$ such that its boundary on the sphere $S^2$ is space-filling! See the following diagram by Thurston \cite[Figure 10]{thurston-bams} that illustrates `a pattern of identification of a circle, here represented as the equator, whose quotient is topologically a sphere. This defines, topologically a sphere-filling curve.'\\

\begin{center}
	
	\includegraphics[height=4cm]{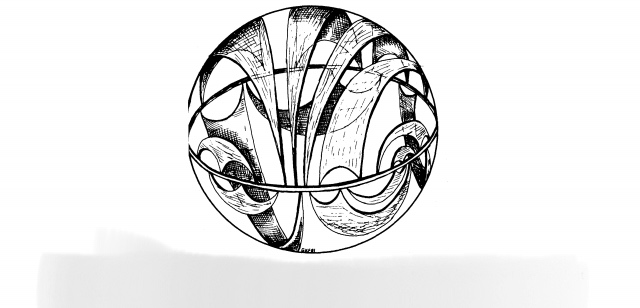}

\end{center}

A version of Question \ref{thurston-qn} was raised by Cannon and Thurston in the  context of closed surface Kleinian groups:

\begin{qn}\cite[Section 6]{CTpub} 
	Suppose that a closed surface group $\pi_1 (S)$ acts freely and properly
	discontinuously on ${\mathbb{H}}^3$ by isometries such that the quotient manifold has no accidental parabolics (here this just means that the image of $\pi_1 (S)$ in $\PSL$ has no parabolics). Does the inclusion
	$\tilde{i} : \widetilde{S} \rightarrow {\mathbb{H}}^3$ extend
	continuously to the boundary?
	\label{ctqn}
\end{qn}

Continuous boundary extensions as in Question \ref{ctqn}, if they exist, are called {\bf Cannon-Thurston maps}.
Question \ref{ctqn} is intimately related to a much older question c.f.\  \cite{abikoff-ct} asking if limit sets are locally connected:

\begin{qn} \label{lcqn} Let $\Gamma$ be a finitely generated Kleinian group such that the limit set $\Lambda_\Gamma$ is connected. Is $\Lambda_\Gamma$ locally connected?
\end{qn}

It is shown in \cite{CTpub}  that for simply degenerate surface Kleinian groups, Questions \ref{ctqn} and \ref{lcqn}
are equivalent, via the Caratheodory extension Theorem. 

The following Theorem of \cite{mahan-split} \cite{mahan-elct} answers questions \ref{ctqn} and \ref{lcqn} affirmatively:

\begin{theorem}\label{ctsurf}
	Let $\rho (\pi_1(S)) =G \subset   \pslc$ be a simply or doubly degenerate (closed) surface Kleinian group. Let $M=\Hyp^3/G$ and $i: S \to M$ be an embedding inducing a homotopy equivalence. Let $\til{i} : \til{S} \to \Hyp^3$ denote a lift of $i$ between universal covers. Let $\mathbb{D}^2, \mathbb{D}^3$ denote the compactifications. Then a Cannon-Thurston map  $\hat{i}: \mathbb{D}^2 \to \mathbb{D}^3$ exists.
	
	Let $\partial i: S^1 \to S^2$ denote the restriction of $\hat{i}$ to the ideal boundaries. Then for $p\neq q$, $\partial i(p) = \partial i(q)$ if and only if $p, q$ are the ideal end-points of a leaf of an ending lamination or ideal end-points of a complementary ideal polygon of an ending lamination.
\end{theorem}

The second part of Theorem 
\ref{ctsurf} shows that the data of the ending lamination can be recovered from the Cannon-Thurston map and so we have an affirmative answer to Question \ref{ctmq}. In conjunction with the Ending Lamination Theorem \ref{elthm}, this establishes the slogan: 

\begin{slogan}
 Dynamics on the limit set determines geometry in the interior. 
\end{slogan}

A number of authors have contributed to the resolution of the above questions.
Initially it was believed \cite{abikoff-ct} that Question \ref{lcqn} had a negative answer for simply degenerate Kleinian groups.
 Floyd \cite{Floyd} proved the corresponding theorem for  geometrically finite Kleinian groups. Then in the early 80's
 Cannon and Thurston
\cite{CT} proved Theorem \ref{ctthm}. This was extended by Minsky \cite{minsky-jams}, Klarreich \cite{klarreich},  Alperin-Dicks-Porti \cite{ADP}, Bowditch \cite{bowditch-stacks, bowditch-ct}, McMullen \cite{ctm-locconn}, Miyachi \cite{miyachi-ct}  and the author \cite{mitra-trees, mahan-bddgeo, mahan-ibdd, mahan-amalgeo}  for various special cases. The general surface group case was accomplished in \cite{mahan-split} and the general Kleinian group case in \cite{mahan-kl}.

\subsection{Geometric Group Theory} We now turn to a generalization of Question \ref{ctqn} to a far more general context. After the introduction of hyperbolic metric spaces by Gromov \cite{gromov-hypgps}, Question \ref{ctqn} was extended by the author \cite{mitra-thesis, bestvinahp, mitra-survey} to the context of a
hyperbolic group $H$ acting freely and properly discontinuously
by isometries on a hyperbolic metric space $X$. Any hyperbolic $X$ has a (Gromov) boundary $\partial X$ given by asymptote-classes of geodesics. Adjoining $\partial X$ to $X$ we get the Gromov
compactification $\hhat X$.

There is a natural map
$i : \Gamma_H \rightarrow X$, sending   vertices of $\Gamma_H$ to
the $H-$orbit of a point $x \in X$, and connecting images of adjacent vertices
by geodesic segments in $X$. Let $\hhat{\Gamma_H}$, $\hhat{X}$ denote the Gromov
compactification of $\Gamma_H$, $X$ respectively.
The analog of Question \ref{ctqn}
 is the following:

\begin{qn}\label{ctq2}
Does $ i : \Gamma_H
\to X$ extend continuously to a  map $\hat{i} : \hhat{\Gamma_H} \to \hhat{X}$?
\end{qn}

Continuous extensions as in Question \ref{ctq2} are also referred to as {\bf Cannon-Thurston maps} and make sense when $\Gamma_H$ is replaced by an arbitrary hyperbolic metric space $Y$. A simple and basic criterion for the existence of Cannon-Thurston maps was established in \cite{mitra-ct, mitra-trees}:
\begin{lemma} Let $i : (Y,y) \to (X,x)$ be a proper map between (based) Gromov-hyperbolic spaces.
	A continuous extension (also called a Cannon-Thurston map)
	 $\hat{i}: \hhat{Y} \to \hhat{X}$
	exists if  and only if the following holds: \\
There exists a non-negative proper function  $M: \natls \to \natls$, such that 
if $\lambda = [a,b]_Y$  is a geodesic lying outside an $N$-ball
	around $y$, then any geodesic segment $[i(a),i(b)]_X$    in $X$ joining $i(a),i(b)$
	 lies outside the $M(N)$-ball around 
	$x=i(y)$.
		\label{crit_hyp}
\end{lemma}

In the generality above Question \ref{ctq2} turns out to have a negative answer.  An explicit counterexample to Question \ref{ctq2} was recently found by Baker and Riley \cite{br-ct} in the context of small cancellation theory. The counterexample  uses Lemma \ref{crit_hyp} to rule out the existence of Cannon-Thurston maps. Further, Matsuda and Oguni \cite{mo} further developed Baker and Riley's counterexample  to show that given a(ny) non-elementary hyperbolic group $H$,  there exists  hyperbolic group $G$ such that $H \subset G$ and there is no Cannon-Thurston map for the inclusion. We shall furnish positive answers to Question \ref{ctq2} in a number of special cases in Section \ref{trees}.

\section{Closed 3-manifolds}
\subsection{3-manifolds fibering over the circle}\label{fiber} We start by giving a sketch of the proof of Theorem \ref{ctthm}, 

The proof is coarse-geometric in nature and follows \cite{mitra-ct, mitra-trees}.
 We recall a couple of basic Lemmata we shall be needing from \cite{mitra-trees}.
The following says that nearest point projection onto a geodesic in a hyperbolic space is coarsely Lipschitz.

\begin{lemma}
	Let $(X,d)$ be a $\delta$-hyperbolic metric space. Then there exists a constant $C \geq 1$ such that the following holds: \\
	Let $\lambda \subset X$ be
	a geodesic segment and let
	$\Pi: X \rightarrow\lambda$ be a nearest point projection. Then $d(\Pi(x),\Pi(y))\leq Cd(x,y)$ for
	all $x, y\in X$.
	\label{npplip}
\end{lemma}

The next Lemma says that  nearest point projections and quasi-isometries almost commute.
\begin{lemma}
	Let $(X,d)$ be a $\delta$-hyperbolic metric space.
	Given $(K,{\epsilon})$, there exists $C$ such that  the following holds: \\
	Let $\lambda = [a,b]$ be a geodesic segment in $X$. Let $p \in X$ be arbitrary and let 
	$q$ be a nearest point projection of $p$ onto $\lambda$.
	Let $\phi$ be a $(K,\epsilon)-$quasi-isometry from $X$ to itself and let $\Phi(\lambda) = [\phi(a),\phi(b)]$ be a geodesic segment in $X$  joining $\phi(a),\phi(b)$. Let
	$r$ be a nearest point projection of $\phi(p)$ onto $\Phi(\lambda)$.
	Then ${d}(r,\phi(q))\leq C$. 
	\label{nppqi}
\end{lemma}

\noindent  {\bf Sketch of Proof:}
The proof of Lemma \ref{nppqi} follows from the fact that a geodesic tripod $T$ (built from $[a,b]$ and $[p,q]$) is quasiconvex in hyperbolic space and that a quasi-isometric image $\phi(T)$ of $T$ lies close to a geodesic tripod $T'$ built from $[\phi(a),\phi(b)]$ and $[\phi(p),r]$. Hence the image $\phi(q)$ of the centroid $q$ of $T$ lies close to the centroid $r$ of $T'$. $\Box$

\subsection{The key tool: hyperbolic ladder}\label{hyplad} The key idea behind the proof of Theorem \ref{ctthm} and its generalizations in \cite{mitra-ct, mitra-trees} is the construction of a {\bf hyperbolic ladder} $\LL_\lambda \subset \til{M}$ for {\it any} geodesic in $\til \s$. The universal cover $ \til{M}$ fibers over $\R$ with fibers $\til \s$. 
Since the context is geometric group theory, we discretize this as follows. Replace $\til \s$ and $\til M$ by quasi-isometric models in the form of Cayley graphs $\G_{\pi_1(\s)}$ and $\G_{\pi_1(M)}$ respectively. Let us denote $\G_{\pi_1(\s)}$ by $Y$ and $\G_{\pi_1(M)}$ by $X$. The projection of $ \til{M}$ to the base $\R$ is discretized accordingly giving a model that can be thought of as (and is quasi-isometric to)  a {\it tree $T$ of spaces}, where 
\begin{enumerate}
	\item $T$ is the simplicial tree underlying $\R$ with vertices at $\Z$.
	\item All the vertex and edge spaces are (intrinsically)  isometric to $Y$.
	\item The edge space to vertex space inclusions are qi-embeddings.
\end{enumerate}

We summarize this by saying that $X$ is a tree $T$ of spaces satisfying the {\it qi-embedded condition} \cite{BF}.

Given a geodesic segment $[a,b]=\lambda = \lambda_0 \subset Y$, we  now sketch the promised construction of 
the {\bf ladder} $\LL_\lambda \subset X$ containing $\lambda$. Index the vertices by $n \in \Z$.
Since the edge-to-vertex space inclusions are quasi-isometries, this induces a quasi-isometry $\phi_n$ from the vertex space $Y_n$ to the vertex space $Y_{n+1}$ for $n\geq 0$. A similar quasi-isometry $\phi_{-n}$ exists from $Y_{-n}$ to the vertex space $Y_{-(n+1)}$. These quasi-isometries are defined on the vertex sets of $Y_n$, $n \in \Z$. The map $\phi_n$ induces a map $\Phi_n$ from geodesic segments in $Y_n$ to geodesic segments in $Y_{n+1}$ for $n\geq 0$ by sending a geodesic in $Y_n$ joining $a,b$ to a geodesic in $Y_{n+1}$ joining $\phi_n(a),\phi_n(b)$. Similarly, for $n \leq 0$.
Inductively define:

\begin{itemize}
	\item	$\lambda_{j+1} = \Phi_j (\lambda_j )$ for $j \geq 0$,
	\item	$\lambda_{-j-1} = \Phi_{-j} (\lambda_{-j} )$ for $j \geq 0$,
	\item 	$\LL_\lambda = \bigcup_j \lambda_j$.
\end{itemize}

$\LL_\lambda$ turns out to be quasiconvex in $X$.
To prove this, we construct a coarsely Lipschitz retraction $\Pi_\lambda: \bigcup_j Y_j \to \LL_\lambda$ as follows.

On $Y_j$ define $\pi_j (y)$ to be a nearest-point projection of $y$ onto $\lambda_j$ and define	
$$\Pi_\lambda (y) = \pi_j (y),  \, for \, y \in Y_j.$$

The following theorem asserts that $\Pi_\lambda$ is coarsely Lipschitz.

\begin{theorem}\cite{mitra-ct, mitra-trees, mahan-bddgeo}
	There exists $C\geq 1$ such that for any geodesic $\lambda \subset Y$,
	$$d_X(\Pi_\lambda (x),{\Pi_\lambda}(y))\leq C d_X(x,y)$$ for $x, y
	\in \bigcup_i Y_i$.
	\label{lipret}
\end{theorem}

\noindent {\bf Sketch of Proof:}\\
The proof requires only the hyperbolicity of $Y$, but not that of $X$. It suffices to show that for $d_X(x,y)=1$, $d_X(\Pi_\lambda (x),{\Pi_\lambda}(y))\leq C$. Thus $x,y$ may be thought of as 
\begin{enumerate}
	\item either lying in the same $Y_j$. This case follows directly from Lemma \ref{npplip}.
	\item or lying vertically one just above the other. Then (up to a bounded amount of error), we can assume without loss of generality, that $y=\phi_j(x)$. This case now follows from Lemma \ref{nppqi}. \hfill $\Box$
\end{enumerate}

Since a coarse Lipschitz retract of a hyperbolic metric space is quasiconvex, we immediately have:

\begin{cor}\label{qcladder}
	If $(X, d_X)$ is hyperbolic, there exists $C\geq 1$ such that for any $\lambda$, $\LL_\lambda$ is $C-$quasiconvex. 
\end{cor}

Note here that we have not used any feature of $Y$ except its hyperbolicity. In particular, we do not need the specific condition that $Y = \til \s$. We are now in a position to prove a generalization of Theorem \ref{ctthm}.

\begin{theorem}\cite{mahan-bddgeo}
	Let $(X,d)$ be a hyperbolic tree ($T$) of hyperbolic metric spaces satisfying the
	qi-embedded condition, where $T$ is 
	$\R$ or $[0, \infty )$ with  vertex and edge sets $Y_j$ as above, 
	$j\in \Z$. Assume (as above) that the edge-to-vertex inclusions are quasi-isometries.
	For	$i : Y_0 \to X$ there is a Cannon-Thurston map ${\hat{{i}}} : \widehat{Y_0}
	\rightarrow \widehat{X}$.
	\label{ctthm2}
\end{theorem}

\begin{proof} Fix a basepoint ${y_0} \in Y_0$. By Lemma \ref{crit_hyp} and quasiconvexity of $\LL_\lambda$
	(Corollary \ref{qcladder}),
	it suffices to show 
	that for all $M\geq{0}$  there exists $N\geq{0}$
	such that if a geodesic segment 
	$\lambda$ lies outside the $N$-ball about $y_0\in{Y_0}$, then
	$\LL_{\lambda}$ lies outside the $M$-ball around ${y_0} \in {X}$. Equivalently, we need a proper function $M(N): \natls \to \natls$.
	
	Since $Y_0$ is properly embedded in $X$, there exists a proper function $g: \natls \to \natls$ such that 
	$\lambda$ lies outside the $g(N)$-ball about $y_0\in X$.
	
	Let $p$ be any point on $\LL_{\lambda}$. Then $p=p_j\in Y_j$ for some $j$. Assume without loss of generality that $j \geq 0$. It is not hard to see that there exists $C_0$, depending only on $X$, such that for any such $p_j$, there exists $p_{j-1} \in Y_{j-1}$ with $d(p_j,p_{j-1}) \leq C_0$. It follows inductively that there exists $y\in \lambda =\lambda_0$
	such that ${d_X}(y,p)\leq C_0j$. Hence, by the triangle inequality, $d_X({y_0},p)  \geq  g(N)-C_0j$.
	
	Next, looking at the `vertical direction', $d_X({y_0},p)  \geq j$ and hence 	$${d_X}({y_0},p)  \geq  \, {\rm max}({g(N)-C_0j}, j) \geq  \frac{g(N)}{A+1}.$$
	
	Defining $M(N) =  \frac{g(N)}{A+1}$, we see that $M(N)$ is a proper function of $N$ and we are done.
\end{proof}

\subsection{Quasiconvexity}\label{qcflow} The structure of Cannon-Thurston maps in Section \ref{fiber} can be used to establish quasiconvexity of certain subgroups of $\pi_1(M)$. Let $H \subset \pi_1(\s)$ be a finitely generated infinite index subgroup of the fiber group. Then, due to the LERF property for surface subgroups, a Theorem of Scott \cite{scott-lerf}, there is a finite sheeted cover where $H$ is geometric, i.e.\ it is carried by a proper embedded subsurface of (a finite sheeted cover of) $\s$. But such a proper subsurface cannot carry a leaf of the stable or unstable foliations $\FF_s$ or $\FF_u$. This gives us the following Theorem of Scott and Swarup:

\begin{theorem}\cite{scottswar} Let $M$ be a closed hyperbolic 3-manifold fibering over the circle with fiber $\s$.
Let $H \subset \pi_1(\s)$ be a finitely generated infinite index subgroup of the fiber group in $\pi_1(M)$.  Then $H$ is quasiconvex in  $\pi_1(M)$.
\label{scottswar}
\end{theorem}

Theorem \ref{scottswar} has been generalized considerably to the context of convex cocompact subgroups of the mapping class group and $Out(F_n)$ by a number of authors \cite{dkl,kdt,dt1,dt2,mahan-rafi}.

\section{Kleinian surface groups: Model Geometries}\label{models}
In this section  we shall describe a sequence of models for degenerate ends of 3-manifolds following \cite{minsky-bddgeom, minsky-jams, mahan-bddgeo, mahan-ibdd, mahan-amalgeo} 
 and \cite{minsky-elc1, minsky-elc2, mahan-split} and indicate  how to generalize the ladder construction of Section \ref{hyplad} incorporating electric geometry \cite{farb-relhyp}. Let $X$ be a hyperbolic metric space, e.g.\ $\Hyp^3$. Let $\HH_X$ be a collection of disjoint convex subsets. Roughly speaking, electrification equips each element of $\HH_X$ with the zero metric, while preserving the metric on $X \setminus (\bigcup_{H \in \HH_X} H)$. The resulting electrified space $\EXH$ is still Gromov hyperbolic under extremely mild conditions and  hyperbolic geodesics in $X$ can be recovered from electric geodesics in the electrified space $\EXH$. This will allow us to establish the existence of Cannon-Thurston maps. We shall focus on closed surfaces and follow the summary in \cite{mahan-lecuire} for the exposition.

The topology of each building block is simple: it is homeomorphic to $S \times [0,1]$, where $S$ is a closed
surface of genus greater than one. Geometrically, 
the top and bottom boundary components in the first three model geometries are uniformly bi-Lipschitz to a fixed hyperbolic structure on $S$. Assume therefore that $S$ is equipped with such a fixed hyperbolic structure. We do so henceforth. The different types of geometries of the blocks make for different model geometries of ends.

\begin{defn} 
	A model $E_m$ is said to be built up of blocks of some prescribed geometries {\bf glued end to end}, if 
	\begin{enumerate}
		\item 
		$E_m $ is homeomorphic to $S \times [0, \infty)$
		\item There exists $L \geq 1$ such that $S \times [i, i+1]$ is  $L-$bi-Lipschitz to a block of one of the three prescribed geometries: bounded, i-bounded or amalgamated (see below).
	\end{enumerate} 
	$S \times [i, i+1]$ will be called the {\bf $(i+1)$th block} of the model $E_m$.
	
	The {\bf thickness} of the \bf $(i+1)$th block is the length of the shortest path between $S \times \{ i \}$
	and $S \times \{ i+1 \}$ in $S \times [i, i+1] (\subset E_m)$.
\end{defn}

\subsection{Bounded geometry}\label{bddgeo}

Minsky \cite{minsky-bddgeom, minsky-jams}
calls an end $E$ of a hyperbolic 3-manifold to be of  bounded geometry
if there are no arbitrarily short
closed geodesics in $E$.

\begin{defn}
	Let $B_0 = S \times [0,1]$ be given the product metric.
	If $B$ is $L-$bi-Lipschitz homeomorphic to $B_0$, it is called an {\bf $L-$thick  block}. 
	
	An end $E$ is said to have a model of bounded geometry if
	there exists $L$ such that
	$E$ is bi-Lipschitz homeomorphic to a model manifold $E_m$ 
	consisting of gluing  $L-$thick blocks end-to-end.
\end{defn}

It follows from work of  Minsky 
\cite{minsky-top}
that if $E$ is of bounded geometry, it has a model of bounded geometry. The existence of Cannon-Thurston maps in this setup is then a replica of the proof of Theorem \ref{ctthm2}.

\subsection{i-bounded Geometry} \label{ibdd}

\begin{defn} \cite{mahan-ibdd}
	An end $E$ of a hyperbolic 3-manifold $M$ has {\bf i-bounded geometry} 
	if the boundary torus of every Margulis tube in $E$ has bounded diameter. 
\end{defn}

We give an alternate description.
Fix a closed hyperbolic surface $S$.  Let
 $\CC$ be  a finite
 collection of (not necessarily disjoint)  simple closed geodesics on
$S$. Let
$N_\epsilon ( \sigma_i )$  denote an $\epsilon$ neighborhood of
$\sigma_i$, $\sigma_i \in \CC, $ were $\ep$ is 
small enough to ensure 
that lifts of  $N_\epsilon ( \sigma_i )$ to 
$\widetilde{S}$ are disjoint.

\begin{defn}\label{ibdddefn}
	Let $I = [0,3]$. Equip $S \times I$ with the product metric. Let
	$B^c = (S \times I - \cup_j N_{\epsilon} ( \sigma_j ) \times
	[1,2])$,
	equipped  with the induced path-metric. Here $\{\sigma_j\}$ is a subcollection of $\CC$ consisting of disjoint curves.
	Perform Dehn filling on some $(1,n)$ curve on each resultant torus component of
	the boundary of $B^c$ (the integers $n$ are quite arbitrary and may vary: we omit subscripts for expository ease). We call $n$ the {\bf twist coefficient}. Foliate the relevant torus boundary component of
	$B^c$ by translates of $(1,n)$ curves. Glue in a 
	solid torus $\Theta$, which we refer to as a {\bf Margulis tube}, with a  hyperbolic metric foliated by totally geodesic 
	disks bounding the $(1,n)$ curves.

	The resulting copy of $S \times I$ thus obtained, equipped with the above metric  is called a {\bf thin  block}. 
\end{defn} 

\begin{defn} A
	model manifold $E_m$ of {\bf i-bounded geometry} is built out of
	gluing  $L-$thick and thin blocks end-to-end (for some $L$) (see  schematic diagram below where the black squares indicate Margulis tubes and the horizontal rectangles indicate the blocks).
\end{defn}

\begin{center}
	
	\includegraphics[height=4cm]{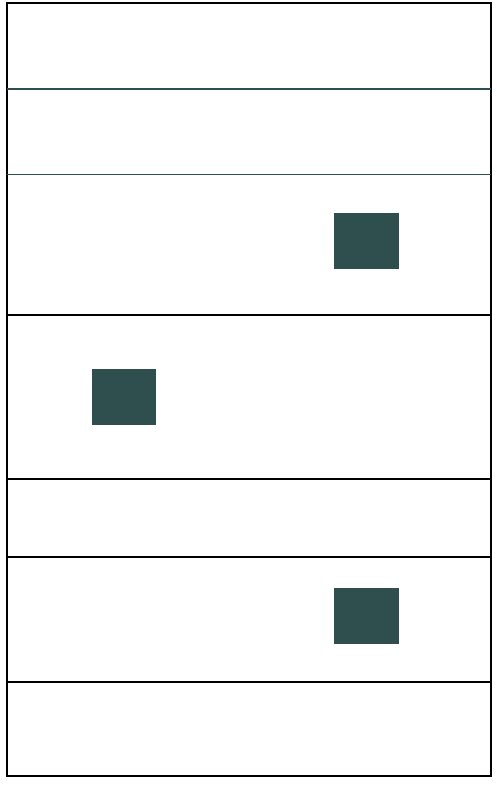}

\end{center} 
It follows from work in \cite{mahan-ibdd} that

\begin{prop}\label{ibdd-model}
	An end $E$ of a hyperbolic 3-manifold $M$ has
	{\bf i-bounded geometry} if and only if
	it is bi-Lipschitz homeomorphic to a  model manifold $E_m$ of i-bounded geometry.
\end{prop}

We give a brief indication of the construction of $\LL_\lambda$ and the proof of the existence of Cannon-Thurston maps in this case. First electrify all the Margulis tubes, i.e.\ equip them with a zero metric (see \cite{farb-relhyp} for details on relative hyperbolicity and electric geometry). This ensures that in the resulting electric geometry, each block is of bounded geometry. More precisely, there is a (metric) product structure on $S \times [0,3]$ such that each $\{x\} \times [0,3]$ has uniformly bounded length in the electric metric.

Further, since the curves in $\CC$ are electrified in a block, Dehn twists are isometries from $S \times \{1\}$ to 
$S \times \{2\}$ in a thin block. This allows the construction of $\LL_\lambda$ to go through as before and ensures that it is quasiconvex in the resulting electric metric. 

Finally given an electric geodesic lying outside large balls modulo Margulis tubes one can recover a genuine hyperbolic geodesic tracking it outside Margulis tubes. A relative version of the criterion of Lemma \ref{crit_hyp} can now be used to prove the existence of Cannon-Thurston maps.

\subsection{Amalgamation Geometry}\label{amalgeo}
Again, as in Definition \ref{ibdddefn}, 
start with a fixed closed hyperbolic surface $S$, a collection of simple closed curves $\CC$ and set $I=[0,3]$. Perform Dehn surgeries on the Margulis tubes corresponding to $\CC$ as before. 
Let $K = S \times [1,2] \subset S \times [0,3]$ and let $K^c = (S \times I - \cup_i N_{\epsilon} ( \sigma_i ) \times [1,2])$. 
Instead of fixing the product metric on the complement $K^c$ of Margulis tubes in $K$,
allow these complementary components to have {\it arbitrary geometry} subject only to the restriction that the geometries of $S \times \{1,2\}$ are fixed. Equip $S \times [0,1]$ and $S \times [2,3]$ with the product metrics.
The resulting block is said to be a {\bf block of  amalgamation geometry}. After lifting to the universal cover, complements of Margulis tubes in the lifts $\til{S} \times [1,2]$ are termed {\bf amalgamation components}. 

\begin{defn}\label{amalgeo-model}
	An end $E$ of a hyperbolic 3-manifold $M$ has
	{\bf amalgamated geometry} if 
	\begin{enumerate}
		\item 	it is bi-Lipschitz homeomorphic to a model manifold $E_m$ 
		consisting of gluing  $L-$thick and amalgamation geometry blocks end-to-end (for some $L$).
		\item Amalgamation components are (uniformly) quasiconvex in $\til E_m$.
	\end{enumerate}
\end{defn}

To construct the ladder $\LL_\lambda$ we electrify amalgamation components as well as Margulis tubes. This ensures that in the electric metric,
\begin{enumerate}
	\item Each amalgamation block has bounded geometry
	\item The mapping class element taking $S \times \{1\}$ to $S \times \{2\}$ induces an isometry of the electrified metrics.
\end{enumerate} 

Quasiconvexity of $\LL_\lambda$ in the electric metric now follows as before. To recover the data of hyperbolic geodesics from quasigeodesics lying close to $\LL_\lambda$, we use (uniform) quasiconvexity of amalgamation components and existence of Cannon-Thurston maps follows.

\subsection{Split Geometry} \label{splitt} We need now to relax the assumption  that the boundary components of model blocks are of (uniformly) bounded geometry. Roughly speaking, split geometry is a generalization of amalgamation geometry where
\begin{enumerate}
	\item A Margulis tube is allowed to travel through a uniformly bounded number of contiguous blocks and split them.
	\item The complementary pieces, now called split components, are quasiconvex in a somewhat weaker sense (see Definition \ref{sg} below).
\end{enumerate}

Each split component is allowed to contain Margulis tubes, called {\bf hanging tubes} that do not go all the way across from the top to the bottom, i.e.\ they
do not split both $\Sigma_i^s, \Sigma_{i+1}^s$.

A split component $B^s \subset B = S \times I$
is  topologically a product $S^s \times I$ for some, necessarily connected
$S^s (\subset S)$. However,  the upper and
lower boundaries of $B^s$ need only 
be split subsurfaces of $S^s$ to allow for hanging tubes
starting or ending (but not both) within the split block.

\begin{center}
	
	\includegraphics[height=4cm]{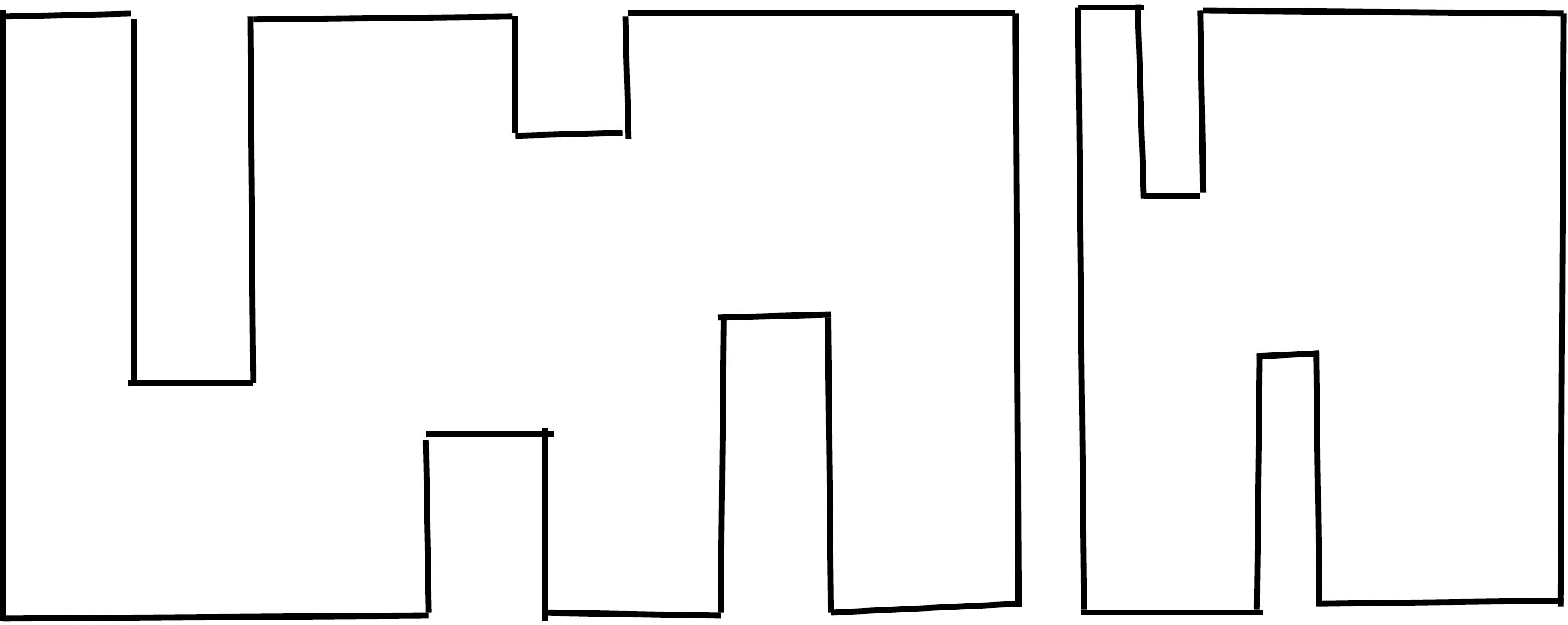}

	{\it Split Block with hanging tubes} 
	
\end{center}
An end $E$ of a hyperbolic 3-manifold $M$ has
{\bf weak split geometry} if 
 	it is bi-Lipschitz homeomorphic to a model manifold $E_m$ 
	consisting of gluing  $L-$thick and split blocks as above end-to-end (for some $L$).
Electrifying split components as in Section \ref{amalgeo}, we obtain a new electric metric called the {\bf graph metric} $d_G$ on $E$. 

\begin{definition}\label{sg}
	A model of weak split geometry is said to be of {\bf split geometry} if the convex hull of each split component has uniformly bounded $d_G-$diameter.
\end{definition}

\section{Cannon-Thurston Maps for Kleinian groups and Applications}
\subsection{Cannon-Thurston maps for degenerate manifolds} Let $M$ be a hyperbolic 3-manifold homotopy equivalent to a closed hyperbolic
surface $S$. Once we establish that $M$ has split geometry, the proof proceeds as in Section \ref{amalgeo} by electrifying split components, constructing a hyperbolic ladder $\LL_\lambda$ and finally recovering a hyperbolic geodesic from an electric one. We shall therefore dwell in this subsection on showing that any degenerate end has split geometry. We shall do this under two simplifying assumptions, directing the reader to \cite{mahan-split} (especially the Introduction) for a more detailed road-map.

We borrow extensively from the hierarchy and model manifold terminology and technology
of Masur-Minsky \cite{masur-minsky2} and Minsky \cite{minsky-elc1}.
The model manifold of \cite{minsky-elc1, minsky-elc2}  furnishes
a resolution, or equivalently, a sequence $\{ P_m \}$ of pants decompositions of $S$ exiting $E$ and hence a {\it hierarchy path}.
Let $\tau_m$ denote the simple multicurve on $S$ constituting  $P_m$.  The $P_m$ in turn  furnish split level surfaces
$\{ S_m \}$ exiting  $E$: a {\it split level surface} $S_m$ is a collection of (nearly) totally geodesic embeddings of the pairs of pants comprising $P_m$. Next, corresponding to the hierarchy path $\{ \tau_m \}$, there is a tight geodesic  in the curve complex $\CC(S)$ of the surface $S$ consisting of the bottom geodesic $\{\eta_i\}$ of the hierarchy. We proceed to extract a subsequence of the resolution $\tau_m$ using the bottom geodesic $\{\eta_i\}$ under two key simplifying assumptions:

\begin{enumerate}
	\item  For all $i$, the length 
	of exactly one curve in $\eta_i$ is sufficiently small, less than the
	Margulis constant in particular. Call it $\eta_i$ for convenience.
	\item Let $S_i$  correspond  to the first occurrence of the vertex $\eta_i$ in the resolution $\tau_m$. Assume further that the $S_i$'s are actually split surfaces and not just split level surfaces, i.e.\ they all  have injectivity radius uniformly bounded below,
\end{enumerate}

It follows that the Margulis tube $\eta_i$   splits both $S_i$ and $S_{i+1}$ and that  the tube $T_i$ is
trapped entirely  between $S_i$ and $S_{i+1}$. 
The product region $B_i$ between 
$S_i$ and $S_{i+1}$ is therefore a   split block for all $i$ and $T_i$ splits it. 
The model manifold thus obtained is one of weak split geometry. In a sense, this is a case of `pure split geometry', where all blocks have a split geometry structure (no thick blocks). To prove that the model is indeed of split geometry, it remains to establish the quasiconvexity condition of Definition \ref{sg}.

Let $K$ be a split component and $\til{K}$ an elevation to $\til E$. Let $v$ be a boundary short curve for the split component and let $T_v$ be the Margulis tube corresponding to $v$ abutting $K$. Denote the
hyperbolic convex hull by $CH(\til{K})$ and pass back to a quotient in $M$. A crucial observation that is needed here is the fact that any pleated surface has bounded $d_G-$diameter. This is because thin parts of pleated surfaces lie inside Margulis tubes that get electrified in the graph metric. It therefore suffices to show that any point in $CH({K})$ lies close to a pleated surface passing near the fixed    tube $T_v$. 
This last condition follows from the Brock-Bromberg drilling theorem \cite{brock-bromberg-density} and the fact that the convex core of a quasi-Fuchsian group is filled by pleated surfaces \cite{fan1}. This completes our sketch of a proof of the following main theorem of \cite{mahan-split}:

\noindent {\bf Theorem \ref{ctsurf}:}
	Let $\rho: \pi_1(S) \to \pslc$ be a simply or doubly degenerate (closed) surface Kleinian group. Then a Cannon-Thurston map exists.
	
	\medskip

It follows that the limit set of $\rho (\pi_1(S))$ is a continuous image of $S^1$ and is therefore locally connected.
As a first application of Theorem \ref{ctsurf},
we shall use the following Theorem of Anderson and Maskit \cite{and-mask} 
to prove that connected limit sets of Kleinian groups without parabolics are locally connected.  

\begin{theorem} \cite{and-mask} Let $\Gamma$ be an analytically finite Kleinian group with connected limit set. Then the limit set $\Lambda (\Gamma )$
	is locally connected if and only if every simply  degenerate surface subgroup of $\Gamma$ without accidental parabolics has locally connected
	limit set.
	\label{and-mask}
\end{theorem}

Combining the remark after Theorem \ref{ctsurf} with Theorem \ref{and-mask}, we  immediately have the following affirmative answer to
Question \ref{lcqn}.

\begin{theorem}
	Let $\Gamma$ be a finitely generated Kleinian group without parabolics and with a connected limit set $\Lambda$. Then $\Lambda$ is locally connected.
	\label{lcfinal}
\end{theorem}

Theorem \ref{ctsurf} can be extended to punctured surfaces \cite{mahan-split} and this allows Theorem \ref{lcfinal} to  be generalized to arbitrary finitely generated Kleinian groups.

\subsection{Finitely generated Kleinian groups}\label{fgkg}
In \cite{mahan-elct}, we show that the point preimages of the Cannon-Thurston map for a simply or doubly
degenerate surface Kleinian group given by Theorem \ref{ctsurf} corresponds to end-points of leaves of ending laminations. In particular, the ending lamination corresponding to a degenerate end can be recovered
from the Cannon-Thurston map. This was extended further in \cite{mahan-red} and \cite{mahan-kl} to obtain the following general version for finitely generated Kleinian groups.

\begin{theorem} \label{ct-kl} \cite{mahan-kl}
	Let $G$ be a finitely generated Kleinian group.	Let $i : \Gamma_G \rightarrow {\mathbb H}^3$ be the natural identification
	of a Cayley graph of $G$ with the orbit of a point in ${\mathbb H}^3$. 
	Further suppose that each degenerate end of $\Hyp^3/G$ can be equipped with a Minsky model \cite{minsky-elc1}.\footnote{This hypothesis is satisfied for all ends without parabolics as well as for ends incompressible away from cusps - see \cite{minsky-elc1,minsky-elc2} and \cite[Appendix]{mahan-kl}. That the hypothesis is satisfied in general would follow from unpublished work of Bowditch \cite{bowditch-elt,bowditch-endinv}.}
	Then $i$ extends continuously to a Cannon-Thurston map 
	$\hat{i}: \hhat{\Gamma_G} \rightarrow {\mathbb D}^3$,
	where $\hhat{\Gamma_G}$ denotes the (relative) hyperbolic compactification of $\Gamma_G$. 
	
	Let $\partial i$ denote the restriction of $\hat{i}$ to the boundary $\partial \Gamma_G$ of $\Gamma_G$.
	Let $E$ be a degenerate end of $N^h= {\mathbb H}^3/G$ and $\til E$ a lift of $E$ to $\til{N^h}$
	and let $M_{gf}$ be an augmented Scott core of $N^h$. Then the ending lamination $\LL_E$ for the end $E$ lifts to a lamination
	on $\til{M_{gf}} \cap \til{E}$. Each such lift $\LL$ of the ending lamination of a degenerate end defines a relation $\RR_\LL$ on the (Gromov) boundary $ \partial \widetilde{M_{gf}}$
	(or equivalently, the relative hyperbolic boundary $\partial_r \Gamma_G$ of $\Gamma_G$),
	given by
	$a\RR_\LL b$ iff $a, b$ are  end-points of a leaf of $\LL$. Let $\{ \RR_i \}$ be the entire collection of relations on $ \partial \widetilde{M_{gf}}$ obtained this way. Let $\RR$ be the transitive closure of
	the union $\bigcup_i \RR_i$. Then $\partial i(a) = \partial i(b)$ iff $a\RR b$.
\end{theorem}

\subsection{Primitive Stable Representations} In  \cite{minsky-primitive} Minsky
introduced   an open subset of the 
$PSL_2({\mathbb C})$ character variety for a free group,  properly containing the Schottky representations,  on which the action of
the outer automorphism group is properly discontinuous. He called these {\bf  primitive stable representations.}
Let $F_n$ be a free group of rank $n$.
An element of $F_n$ is  primitive if it is an element of a free generating set.  Let $\PP = \cdots www\cdots$ be the set of bi-infinite words with $w$ cyclically reduced  primitive.
A representation $\rho: F_n \rightarrow PSL_2(\mathbb{C})$  is primitive
stable if all elements of $\PP$ are mapped to uniform quasigeodesics in ${\mathbb{H}}^3$.

Minsky conjectured that primitive stable representations are characterized by 
the feature that every component of the ending lamination is blocking.

Using Theorem \ref{ct-kl},   Jeon and Kim \cite{woojin}, and  Jeon, Kim, Ohshika and Lecuire \cite{jkol} resolved  this conjecture.
We briefly sketch   their  argument for a degenerate free Kleinian group without parabolics.

Let  $\{ D_1, \cdots , D_n \}=\DD$  be a finite set of essential  disks on a handlebody
$H$ cutting $H$ into a 3-ball. A free generating set of 	$ F_n$  is dual to $\DD$. For a lamination $\LL$, the {\bf Whitehead graph}	$Wh(\Lambda , \DD)$ is defined as follows.  Cut $\partial H$ along $\partial \DD$ to obtain a  sphere with $2n$ holes,  labeled by $D_i^\pm$. The
vertices of $Wh(\LL , \DD)$ are the  boundary circles of $\partial H$, with an edge whenever
two circles are joined by an arc of $\LL \setminus \DD$. For the ending lamination $\LL_E$ of a degenerate free group
without parabolics, $Wh(\LL_E , \Delta)$  is connected
and has no cutpoints.

Let $\rho_E$ be the associated representation.
If  $\rho$  is not primitive stable, then there exists a sequence of primitive cyclically reduced elements $w_n$ such
that $\rho(w_n^\ast)$ is not an $n-$ quasi-geodesic. 
After passing to a subsequence, $w_n$ and hence $w_n^\ast$ converges to a bi-infinite geodesic $ w_\infty$ in the Cayley graph with two distinct end points $w_+, w_-$  in the Gromov boundary of $F_n$. The Cannon-Thurston map identifies $w_+, w_-$.
Hence by Theorem \ref{ct-kl} they are  either the end points of  a leaf of $\LL_E$  or ideal end-points of a complementary ideal
polygon of $\LL_E$. It follows therefore that $Wh( w_\infty , \DD)$  is connected
and has no cutpoints. Since $w_n$'s converge to $w_\infty$,  $Wh( w_n , \DD)$  is connected
and has no cutpoints for  large enough $n$.  A Lemma due to  Whitehead says that if $Wh( w_n , \DD)$  is connected
and has no cutpoints, then 
$w_n$ cannot be primitive, a contradiction.

\subsection{Discreteness of Commensurators}
In \cite{llr} and \cite{mahan-commens}, Theorems \ref{ctsurf} and \ref{ct-kl} are used to prove that commensurators of finitely generated,  infinite covolume, Zariski dense Kleinian
groups are discrete. The basic fact that goes into the proof is that commensurators preserve the structure of point pre-images of Cannon-Thurston maps. The point pre-image structure is known from Theorem  \ref{ct-kl}.

\subsection{Radial and horospherical limit sets}

\begin{defn}
	A point $\xi \in S^2$ is a radial or conical limit point of  Kleinian group $\G$, if for any base-point $o\in \til{M}$
	and any geodesic $\gamma_\xi$ ending at $\xi$, there exist $C \geq 0$ and infinitely many translates $g.o \in N_C(\gamma_\xi)$,
	$g \in \G$.

	A point $\xi \in S^2$ is a horospherical limit point of  Kleinian group $\G$, if for any base-point $o\in \til{M}$
	and any horoball $B_\xi$ based at $\xi$, there exist infinitely many translates $g.o \in B_\xi$,
	$g \in \G$.
	
	The collection of radial (resp. horospherical) limit points of $\G$ is called the radial (resp. horospherical) limit set of $\G$ and is denoted by $\Lambda_r$ (resp. $\Lambda_h$).
	
	The  multiple 
	limit set $\Lambda_m$ consists of those point of $S^2$ that have more than one pre-image under the Cannon-Thurston map. 
\end{defn}

Several authors \cite{kapovich-ahlfors, gerasimov, jklo} worked on the relationship between  $\Lambda_m$ and $\Lambda_r$. They concluded that the conical limit set is strictly contained in the set of 
injective points of the Cannon-Thurston map, i.e.\ $\Lambda_r \subset \Lambda_m^c$, but the inclusion is proper.

In \cite{mahan-lecuire}, we showed:
\begin{theorem} \label{hll}
	$\Lambda_m^c=\Lambda_h$. 
\end{theorem}

\subsection{Motions of limit sets } We discuss the following question in this section, which  paraphrases the second part of   of \cite[Problem 14]{thurston-bams}.
A detailed survey appears in \cite{mahan-cmsurvey}.

\begin{qn}\label{motionsq}
	Let $G_n$ be a sequence of Kleinian groups converging to a Kleinian group $G$. Does the corresponding dynamics of $G_n$ on the Riemann sphere $S^2$  converge to the dynamics  of $G$ on  $S^2$?
\end{qn}

To make Question \ref{motionsq} precise, we need to make sense of `convergence' both for Kleinian groups and for their dynamics on $S^2$. There are three different notions of convergence for Kleinian groups.  

\begin{definition}\label{topofconv} Let $\rho_i : H \to \pslc$ be 
	a sequence of Kleinian groups. We say that that $\rho_i$ converges to $\rho_\infty$ {\bf algebraically} if for all $h\in H$, $\rho_i(h) \to \rho_\infty (h)$.
	
	Let $\rho_j: H \rightarrow \pslc$ be a sequence of discrete, faithful representations of a finitely generated, torsion-free, nonabelian group $H$. If  $\{\rho_j(H)\}$ converges as a sequence of closed subsets of $\pslc$
	to a torsion-free, nonabelian Kleinian group
	$\G$,  $\G$ is called the 
	geometric limit of the sequence.
	
	$G_i (= \rho_i(H))$ converges {\bf strongly }
	to  $G (= \rho_\infty(H))$ if the convergence is both geometric and algebraic.
\end{definition}

Question \ref{motionsq} then splits into the following three questions.
\begin{qn}\label{q2}
	\begin{enumerate}
		\item If $G_n \to G$ geometrically, then do the corresponding limit sets converge in the Hausdorff topology on $S^2$?
		\item If $G_n \to G$ strongly then do the corresponding Cannon-Thurston maps converge uniformly?
		\item If $G_n \to G$ algebraically then do the corresponding Cannon-Thurston maps converge pointwise?
	\end{enumerate}
\end{qn}
We give the answers straight off and then proceed to elaborate.
\begin{ans}
	\begin{enumerate}
		\item The answer to Question \ref{q2} (1) is Yes.
		\item The answer to Question \ref{q2} (2) is Yes.
		\item The answer to Question \ref{q2} (3) is No, in general.
	\end{enumerate}
\end{ans}

The most general answer to Question \ref{q2} (1) is due to Evans \cite{evans1}, \cite{evans2}:

\begin{theorem} \cite{evans1}, \cite{evans2} Let $\rho_n : H \rightarrow G_n$ be a sequence of weakly type-preserving
	isomorphisms from a geometrically finite group
	$H$ to Kleinian groups $G_n$ with limit sets $\Lambda_n$, 
	such that $\rho_n$ converges algebraically to $\rho_\infty : H \rightarrow G^a_\infty$ and geometrically to $G^g_\infty$. Let $\Lambda_a$ and
	$\Lambda_g$ denote the limit sets of $G^a_\infty$ and  $G^g_\infty$.
	Then $\Lambda_n \rightarrow \Lambda_g$ in the Hausdorff metric. Further, the sequence converges strongly if and only 
	$\Lambda_n \rightarrow \Lambda_a$ in the Hausdorff metric.
	\label{evans}
\end{theorem}

The answer to Question \ref{q2} (2) is due to the author and Series \cite{mahan-series2} in the case that $H=\pi_1(S)$ for a closed surface $S$ of genus greater than one. This can be generalized to arbitrary finitely generated Kleinian groups as in \cite{mahan-cmsurvey}:

\begin{theorem}
	Let $H$ be a fixed group and $\rho_n (H ) = \G_n$ be a sequence
	of geometrically finite Kleinian groups converging strongly to a Kleinian group $\G$. 
	Let $M_n$ and $M_\infty$ be the corresponding hyperbolic manifolds. Let $K$ be a fixed complex with fundamental group $H$.
	
	Consider
	embeddings $\phi_n : K \rightarrow M_n, n = 1, \cdots , \infty$
	such that the maps $\phi_n$ are homotopic to each other by uniformly bounded homotopies (in the geometric limit).  
	Then Cannon-Thurston maps for  $\til{\phi_n}$ exist and converge uniformly to the Cannon-Thurston map for 
	$\til{\phi_\infty}$.
\end{theorem}  

Finally we turn to Question \ref{q2} (3), which turns out to be the subtlest. In \cite{mahan-series1} we showed that the answer to Question \ref{q2} (3) is `Yes' if the geometric limit is geometrically finite.
We illustrate this with a concrete example due to Kerckhoff and Thurston  \cite{ kerckhoff-thurston}

\begin{theorem}
	Fix a closed hyperbolic surface $S$ and a simple closed geodesic $\sigma$ on it. Let $tw^i$ denote the automorphism of $S$ given by
	an $i$-fold Dehn twist along $\sigma$. Let $G_i$ be the quasi-Fuchsian group given by the simultaneous uniformization of $(S,tw^i(S))$. Let $G_\infty$
	denote the geometric limit of the $G_i$'s. Let
	$S_{i-}$ denote the lower boundary component of the convex core of $G_i$, $i = 1, \cdots , \infty$ (including $\infty$).
	Let $\phi_i:S \rightarrow S_{i-}$ be such that if  $0 \in {\mathbb{H}}^2 = \widetilde{S}$ denotes the origin of ${\mathbb{H}}^2$ then
	$\widetilde{\phi_i} (0)$ lies in a uniformly bounded neighborhood of $0 \in {\mathbb{H}}^3 = \widetilde{M_i}$. We also assume (using the fact that
	$M_\infty$ is a geometric limit of $M_i$'s) that $S_{i-}$'s converge geometrically to $S_{\infty -}$. Then the Cannon-Thurston maps
	for $\widetilde{\phi_i}$ converge pointwise, but not uniformly, on $\partial {\mathbb{H}}^2$ to the Cannon-Thurston map
	for $\widetilde{\phi_\infty}$.
\end{theorem}

However, if the geometric limit is geometrically infinite, then the answer to Question \ref{q2} (3) may be negative. We illustrate this with certain examples of geometric limits constructed by
Brock in \cite{brock-itn}.

\begin{theorem} \cite{mahan-series2}\label{brockct}
	Fix a closed hyperbolic surface $S$ and a separating simple closed geodesic $\sigma$ on it, cutting $S$ up into two pieces $S_{-}$ and $S_{+}$. 
	Let $\phi$ denote an automorphism of $S$ such that $\phi |_{S_{-}}$ is the identity and $\phi |_{S_{+}} = \psi$ is a pseudo-Anosov of $S_{+}$ fixing the boundary. Let $G_i$ be the quasi-Fuchsian group given by the simultaneous uniformization of $(S,\phi^i(S))$. Let $G_\infty$
	denote the geometric limit of the $G_i$'s. Let
	$S_{i0}$ denote the lower boundary component of the convex core of $G_i$, $i = 1, \cdots , \infty$ (including $\infty$).
	Let $\phi_i:S \rightarrow S_{i0}$ be such that if  $0 \in {\mathbb{H}}^2 = \widetilde{S}$ denotes the origin of ${\mathbb{H}}^2$ then
	$\widetilde{\phi_i} (0)$ lies in a uniformly bounded neighborhood of $0 \in {\mathbb{H}}^3 = \widetilde{M_i}$. We also assume (using the fact that
	$M_\infty$ is a geometric limit of $M_i$'s) that $S_{i0}$'s converge geometrically to $S_{\infty 0}$. 
	
	Let $\Sigma$ be a complete hyperbolic structure on $S_{+}$ such that $\sigma$ is homotopic to a cusp on $\Sigma$.
	Let $\LL$ consist of pairs $(\xi_{-}, \xi )$ of ideal endpoints (on ${\mathbb{S}}^1_\infty$) of stable leaves $\lambda$ of the stable lamination of $\psi$ acting on 
	$\til{\Sigma}$. Also let $\partial \til{\mathcal{H}}$ denote the collection of ideal basepoints of horodisks given by lifts (contained in  $\til{\Sigma}$)
	of the cusp in $\Sigma$ corresponding to $\sigma$.
	Let 
	
	\begin{center}
		$\Xi = \{ \xi :$ There exists $\xi_{-}$ such that $(\xi_{-}, \xi ) \in \LL ; \xi_{-} \in  \partial \til{\mathcal{H}} \}.$
	\end{center}
	
	Let $\partial \phi_i$, $i= 1 \cdots , \infty$ denote the Cannon-Thurston maps for $\til{\phi_i}$. Then
	
	\begin{enumerate}
		\item $\partial \phi_i (\xi )$ does not converge to $\partial \phi_\infty (\xi )$ for $\xi \in \Xi$.
		\item $\partial \phi_i (\xi )$  converges to $\partial \phi_\infty (\xi )$ for $\xi \notin \Xi$.
	\end{enumerate}
\end{theorem}

In \cite{mahan-ohshika}, we identify the exact criteria that lead to the discontinuity phenomenon of Theorem \ref{brockct}.

\section{Gromov-Hyperbolic groups}\label{trees}

\subsection{Applications and Generalizations}\label{normaltrees}
\subsubsection{Normal subgroups and trees}
The ladder construction of Section \ref{hyplad} has been generalized considerably. We work in the context of an exact sequence $1 \to N \to G \to Q \to 1$, with $N$ hyperbolic and $G$ finitely presented. 
We observe that for the proof of Theorem \ref{ctthm2} to go through it suffices to have a qi-section of $Q$ into $G$ to provide a `coarse transversal' to flow. Such a qi-section was shown to exist by Mosher \cite{mosher-hypextns}. We then obtain the following.

\begin{theorem}\cite{mitra-ct}
	Let $G$ be a hyperbolic group and let $H$ be a hyperbolic normal subgroup
	that is normal in $G$. Then the inclusion of Cayley graphs
	$i : \Gamma_H\rightarrow\Gamma_G$ gives a Cannon-Thurston 	map $\hat{i}
	: \widehat{\Gamma_H}\to \widehat{\Gamma_G}$.
	\label{cthypext}
\end{theorem}

The ladder construction can also be generalized to the general framework of a tree of hyperbolic metric spaces.

\begin{theorem}\cite{mitra-trees}
	Let $(X,d)$ be a tree ($T$) of hyperbolic metric spaces satisfying the
	qi-embedded condition (i.e.\ edge space to vertex space inclusions are qi-embeddings).  Let $v$ be a vertex of $T$
	and $({X_v},d_v)$ be the vertex space corresponding to $v$.
	If $X$ is hyperbolic then the inclusion  $i: X_v \to X$ gives a Cannon-Thurston map $\hat{i}: \hhat{X_v} \to \hhat{X}$.
	\label{tree}
\end{theorem}

Theorem \ref{cthypext} was generalized by the author and Sardar to a purely coarse geometric context, where no group action is present. The relevant notion is that of a metric bundle for which we refer the reader to \cite{mahan-sardar}. Roughly speaking, the data of a metric bundle consists of vertex and edge spaces as in the case of a tree of spaces, with two notable changes:
\begin{enumerate}
\item The base $T$ is replaced by an arbitrary graph $B$.
\item All edge-space to vertex space maps are quasi-isometries rather than just quasi-isometric embeddings.
\end{enumerate}

With these modifications in place we have the following generalizations of Mosher's qi-section Lemma \cite{mosher-hypextns} and Theorem \ref{cthypext}:

\begin{theorem}  \label{ctbundle}\cite{mahan-sardar}
	Suppose $p:X\rightarrow B$ is a metric graph bundle satisfying the following:
	\begin{enumerate}
	\item  $B$ is a Gromov hyperbolic graph.
	\item Each  fiber $F_b$, for $b$ a vertex of $ B$  is  $\delta$-hyperbolic (for some $\delta >0$) with respect to the  path metric induced from $X$.
	\item The barycenter maps $\partial^3F_b \to F_b$, $b\in B$, sending a triple of distinct points on the boundary $\partial F_b$ to their centroid, are (uniformly, independent of $b$) coarsely  surjective.
	\item $X$ is hyperbolic.
	\end{enumerate}

	Then there is a qi-section $B \to X$.
	The inclusion map $i_b: F_b \rightarrow X$ gives a Cannon-Thurston map
	$\hat{i} : \hhat{F_b} \to \hhat{X}$.
\end{theorem}

\subsection{Point pre-images: Laminations}
In Section \ref{fiber}, it was pointed out that the Cannon-Thurston map $\hat i$ identifies $p, q \in S^1$ if and only if $p, q$ are end-points of a leaf or ideal end-points of a complementary ideal polygon of the stable or unstable lamination. 

In \cite{mitra-endlam} an algebraic theory of ending laminations was developed based on 
 Thurston's theory \cite{thurstonnotes}.
The theory was developed in the context of 
a normal hyperbolic 
subgroup of a  hyperbolic group $G$ and used to give an explicit structure for the Cannon-Thurston map in Theorem \ref{cthypext}.  

\begin{defn}\cite{BFH-lam, chl07, chl08a, chl08b, kl10, kl-jlms, mitra-endlam}
An {\bf algebraic lamination}  for a hyperbolic group $H$ is an
$H$-invariant, flip invariant, closed subset $\LL \subseteq \partial^2
H =(\partial H \times \partial H \setminus \Delta)/\sim$, where $(x,y)\sim(y,x)$ denotes the flip and $\Delta$ the diagonal in $\partial H \times \partial H$.
\end{defn}

Let $$1 \rightarrow H \rightarrow G \rightarrow Q \rightarrow 1$$
be an exact sequence of finitely presented groups with $H$, $G$ hyperbolic. It follows by work of Mosher \cite{mosher-hypextns}
that $Q$  is hyperbolic.  In  \cite{mitra-endlam}, we construct algebraic  
ending laminations 
naturally parametrized by points in the boundary $\partial Q$. We describe the construction now.

Every element $g\in G$ gives an automorphism 
of $H$ sending $h$ to $g^{-1}hg$ for all $h\in H$. Let  
 $\phi_g : \VV (\Gamma_H) \to \VV (\Gamma_H)$ be the resulting bijection of the vertex set  $ \VV (\Gamma_H)$ of $\Gamma_H$. This induces  a map $\Phi_g$
sending an edge $[a,b] \subset \G_H$  to a geodesic segment
joining $\phi_g (a), \phi_g (b)$. 

For some (any) $z\in{\partial{\Gamma_Q}}$ we shall describe an algebraic ending lamination $\Lambda_z$.
Fix such a $z$ and let 
\begin{enumerate}
\item $[1,z) \subset \G_Q$ be a geodesic ray starting at 
$1$ and converging to $z\in{\partial}{\G_Q}$.
\item $\sigma: Q\to G$ be a qi section.
\item $z_n$ be the vertex on $[1,z)$ such that ${d_Q}(1,{z_n}) = n$.
\item ${g_n} = {\sigma}({z_n})$.
\end{enumerate}

For $h\in H$, let ${\SSS}_n^h$ be the $H$--invariant collection of
all free homotopy representatives (or equivalently,  shortest representatives in the
same conjugacy class) of  ${\phi}_{g_n^{-1}}(h)$ in $\G_H$.  
Identifying equivalent geodesics in $\SSS_n^h$ one obtains a subset
$\mathbb{S}_n^h$ of unordered pairs of points in 
${\hhat{{\G}_H}}$. The intersection 
with ${\partial}^{2}{H}$ of the union
of all subsequential limits (in the Hausdorff topology)
of $\{{\mathbb{S}_n^h }\}$ is denoted by 
${\Lambda}_z^h$. 

\begin{defn} 	The set of {\bf algebraic  ending laminations corresponding to
	$z\in{\partial}{\Gamma_Q}$}  is given by
	$$\lel (z) = {{\bigcup}_{h\in{H}}}{{\Lambda}_z^h}.$$
\end{defn} 

\begin{defn}The set $\Lambda$ 
	of all {\it algebraic ending laminations} is defined
	by 
	$$\lel = {{\bigcup}_{z\in{\partial}{\Gamma_Q}}}{{\lel(z)}}.$$
\end{defn}

The following was shown in \cite{mitra-endlam}:

\begin{theorem}\label{ptpre-alg}
The Cannon-Thurston map 
$\hat{i}$ of Theorem \ref{cthypext}
identifies the end-points of 
leaves of $\lel$. Conversely,
if $\hat{i}(p) = \hat{i}(q)$ for
 $p \neq q \in \partial\Gamma_H$, then some bi-infinite geodesic
having $p,q$ as its end-points is a leaf of $\lel$.
\end{theorem}

\subsubsection{Finite-to-one} The classical Cannon-Thurston map of Theorem \ref{ctthm} is finite-to-one.
Swarup asked (cf.\ Bestvina's Geometric Group Theory problem list \cite[Prolem 1.20]{bestvinahp}) if the Cannon-Thurston maps of Theorem \ref{cthypext} are also finite-to-one. Kapovich and Lustig answered this in the affirmative in the following case.

\begin{theorem}\label{kl}\cite{kl-jlms}
	Let 
	$\phi\in Out(F_N)$ be a fully irreducible hyperbolic automorphism.
	Let $G_\phi=F_N\rtimes_\phi \mathbb Z$ be the associated mapping torus group. Let $\partial i$ denote the Cannon-Thurston map of Theorem \ref{cthypext} in this case.
	Then for every $z \in \partial G_\phi$, the cardinality of $(\partial i)^{-1}(z)$ is at most
$2N$.
\end{theorem}

Bestvina-Feighn-Handel \cite{BFH-lam} define a closely related set $\Lambda_{BFH}$ of algebraic laminations in the case covered by Theorem \ref{kl} using train-track representatives of free group automorphisms. Any algebraic lamination
$\LL$ defines a relation $\RR_\LL$ on $\partial F_N$ by $a \RR_\LL b$ if $(a,b) \in \LL$. The  transitive closure of $\LL$ will be called its diagonal closure. In \cite{kl-jlms}, Kapovich and Lustig further show that in the case covered by Theorem \ref{kl}, $\lel$ precisely equals the diagonal closure of $\Lambda_{BFH}$.

\subsection{Relative hyperbolicity}
The notion of a Cannon-Thurston map can be extended to the context of relative hyperbolicity. This was done in \cite{mj-pal}.
Let $X$ and $Y$ be  relatively hyperbolic spaces,  hyperbolic relative to the collections $\HH_X$ and $\HH_Y$ of `horosphere-like sets' respectively. Let us denote the {\it horoballifications} of $X$ and $Y$ with respect to $\HH_X$ and $\HH_Y$ by $\GXH, \GYH$ respectively (see \cite{bowditch-relhyp} for details).  The horoballification of an $H$ in $\HH_X$ or $\HH_Y$ is denoted as $H^h$. Note that $\GXH, \GYH$ are hyperbolic.
The electrifications will be denoted as $\EXH, \EYH$.

\begin{defn}
A map $i: Y \to X$ is strictly type-preserving if 
\begin{enumerate}
\item for all $H_Y\in \HH_Y$ there exists $H_X\in \HH_X$ such that $i(H_Y)\subset H_X$, and 
\item images of distinct horospheres-like sets in $Y$   lie in distinct horosphere-like sets in $X$.
\end{enumerate}
\end{defn}
 Let $i: Y\to X$ be a strictly type-preserving proper embedding.  Then $i$  induces a proper embedding  $i_h\colon \GYH \to \GXH$. 

\begin{defn}
	A Cannon-Thurston map  exists for a strictly type-preserving inclusion $i: Y \rightarrow X$ of relatively hyperbolic
	spaces   if a Cannon-Thurston map  exists for the induced map $i_h\colon \GYH\to \GXH$.
\end{defn}

Lemma \ref{crit_hyp} generalizes to the following.

\begin{lemma}\label{crit-relhyp}
	A Cannon-Thurston map for $i: Y \to X$ exists if and only if
	there exists a non-negative proper function  $M: \natls \to \natls$  such that the
	following holds: \\
	Fix a base-point $y_0\in Y$.  Let $\hat \lambda$ in $\EYH$ be
	an electric geodesic segment  starting and ending  outside horospheres.
	If  $\lambda^b = \hat \lambda \setminus \bigcup_{K \in \HH_Y} K$
	lies outside  $B_N (y_0) \subset Y$,
	then for any electric quasigeodesic $\hat \beta$ joining the
	end points of $\hat i (\hat \lambda)$ in $\EXH$,
	$\beta ^b = \hat \beta \setminus \bigcup_{H \in \HH_X} H$ lies outside
	$B_{M(N)} (i(y_0)) \subset X$.
\end{lemma}

Theorem \ref{tree} then generalizes to:
\begin{theorem}\cite{mj-pal}\label{mjpal}
Let $P\colon X\to T$ be a tree  of relatively hyperbolic spaces satisfying the qi-embedded condition. Assume  that
\begin{enumerate}
\item the inclusion maps of edge-spaces into vertex spaces are strictly type-preserving
\item the induced tree of electrified (coned-off) spaces continues to satisfy the qi-embedded condition
\item $X$  is 
	strongly hyperbolic relative to the family $\CC$ of maximal
cone-subtrees of horosphere-like sets.
\end{enumerate}
Then a Cannon-Thurston map exists for the inclusion of relatively hyperbolic spaces
$i\colon X_{v}\to X$, where
$({X_v},d_{X_v})$ is the relatively hyperbolic vertex space corresponding to $v$.
\end{theorem}

\subsection{Problems} The above survey is conditioned and limited by the author's bias on the one hand and space considerations on the other. In particular we have omitted the important work on quasigeodesic foliations by Calegari, Fenley and Frankel \cite{calegari-u,calegari-r,fenley,frankel} and the topology of ending lamination spaces by Gabai and others \cite{gabai1,gabai2,lms,pryz} as this would be beyond the scope of the present article. A more detailed survey appears in \cite{mahan-ctsurvey}. We end with some open problems.

\subsubsection{Higher dimensional Kleinian groups} As a test-case we propose:

\begin{qn}\label{ct-hid}
Let $S$ be a closed surface of genus greater than one and let $\Ga= \pi_1(S)$ act freely, properly discontinuously by isometries on $\Hyp^n$, $n>3$ (or more generally a rank one symmetric space). Does a Cannon-Thurston map exist in general?
\end{qn}

Can the small cancellation group of Baker-Riley in  \cite{br-ct} act geometrically on a rank one symmetric space thus giving a negative answer to Question \ref{ct-hid} with surface group replaced by free group? Work of Wise \cite{wise,wise-sc} guarantees linearity of such small cancellation groups.

The critical problem in trying to answer Question \ref{ct-hid} is the absence of new examples in higher dimensions. It would be good to find new examples or prove that they do not exist. A version of a question due to Kapovich \cite{kapovich-hid} makes this more precise and indicates our current state of knowledge/ignorance:

\begin{qn}\label{ct-hidsp}
	Let $S$ be a closed surface of genus greater than one and let $ \Ga $ be a discrete subgroup of $ SO(n,1)$ (or more generally a rank one  Lie group)) abstractly isomorphic to  $\pi_1(S))$ acting by isometries on $\Hyp^n$, $n>3$ (more generally the associated symmetric space) such that
	\begin{enumerate}
		\item 	 orbits are not quasiconvex,
		\item 	no element of $\Ga$ is a parabolic.
	\end{enumerate}
 Does $\rho$ factor through a representation to a simply or doubly degenerate (3-dimensional) Kleinian group followed by a deformation of $SO(3,1)$ in $SO(n,1)$?
\end{qn}

A closely related folklore question asks:
\begin{qn}\label{qfiber}
Can a  closed higher dimensional $n>3$ rank one manifold fiber? In particular over the circle?
\end{qn}

It is known, from the Chern-Gauss-Bonnet theorem that a $2n$ dimensional rank one manifold cannot fiber over the circle. Unpublished work of Kapovich \cite{kap-ns} shows that a complex hyperbolic 4-manifold cannot fiber over a 2-manifold.

\subsubsection{Surface groups in higher rank} A topic of considerable current interest is higher dimensional Teichm\"uller theory and Anosov representations of surface groups \cite{labourie,klp,ggkw}. Kapovich, Leeb and Porti give an equivalent definition of Anosov representations in purely coarse geometric terms as representations that are {\bf  asymptotic embeddings}. It will take us too far afield to define these notions precisely here.  What we will say however is that if $\rho : \pi_1(S) \to \GG$ is a discrete faithful representation into a semi-simple Lie group $\GG$ and $\GG/\PP = \BB$ is the Furstenberg boundary, then the Anosov property implies that an orbit $\rho ( \pi_1(S)).o$ `extends' to a $\rho ( \pi_1(S))-$equivariant embedding of $\partial \Ga_{\pi_1(S)} (=S^1)$ into $\BB$.  Thus the boundary map $\Delta: S^1 \to \BB$ maybe thought of as a {\it higher rank Cannon-Thurston map}. For the representation to be Anosov, $\Delta$ is thus an embedding \cite{klp}.

\begin{qn}\label{hirqn}
What class of representations do we get if we require only that $\Delta$ is continuous?
\end{qn}

Question \ref{hirqn} is basically asking for a `nice' characterization of representations that admit a higher rank Cannon-Thurston map. The core problem in addressing it again boils down to finding some rich class of examples. Question \ref{ct-hidsp} has a natural generalization to this context where we replace $SO(n,1)$ by $\GG$.

\subsubsection{Geometric group theory} As we have seen in Section \ref{normaltrees}, normal hyperbolic subgroups and trees of spaces provide  examples where there is a positive answer to Question \ref{ctq2}. Some sporadic new examples have also been found, e.g. hydra groups \cite{br-hydra}. However no systematic theory exists.
In the light of the counterexample in \cite{br-ct}, the general answer to Question \ref{ctq2} is negative. Are there necessary and/or sufficient conditions beyond Lemma \ref{crit_hyp} to guarantee existence of Cannon-Thurston maps? 

As illustrated in \cite{mitra-trees} and \cite{br-hydra}, distortion of subgroups \cite{gromov-ai} is irrelevant. Distortion captures the relationship between $d_H(1,h)$ with $d_G(1,h)$.  On the other hand Cannon-Thurston maps capture the corresponding relationship between  $d_H(1,[h_1,h_2]_H)$ with $d_G(1,[h_1,h_2]_G)$, i.e.\ existence of  Cannon-Thurston maps is equivalent to a proper embedding of `pairs of points' (coding geodesic segments).  The function associated with such a  proper embedding is closely related to the modulus of continuity of the Cannon-Thurston map \cite{br-hydra}.

\section*{Acknowledgments} I would like to thank Benson Farb and Thomas Koberda for several helpful comments on an earlier draft.

\bibliography{ct_icm}
\bibliographystyle{alpha}

%%%%%%%%%%%%%%%%%%%%%%%%%%%%%%%%%%%%%%%%%%%%%%%%%%%%%%%%%%%%%%%%%%%%%%
\end{document}